\documentclass[11pt]{amsart}
\usepackage{amssymb}
\usepackage{amsmath,amssymb}
\usepackage{color}
\usepackage[all]{xy}
\usepackage{tikz}
\usepackage{tikz-cd}
\usepackage{longtable}
\usepackage{graphicx}
\usepackage{cite}
\usepackage{enumitem}
\usetikzlibrary{calc,arrows,backgrounds,decorations.markings, shapes.misc}

\theoremstyle{plain}
\newtheorem{thm}{Theorem}[section]
\newtheorem{theorem}[thm]{Theorem}

\newtheorem{lemma}[thm]{Lemma}
\newtheorem{corollary}[thm]{Corollary}
\newtheorem{proposition}[thm]{Proposition}
\theoremstyle{definition}
\newtheorem{remark}[thm]{Remark}

\newtheorem{definition}[thm]{Definition}

\newtheorem{example}[thm]{Example}

\numberwithin{equation}{section}

\newcommand{\sM}{{\mathcal M}}
\newcommand{\sN}{{\mathcal N}}
\newcommand{\sO}{{\mathcal O}}

\newcommand{\sS}{{\mathcal S}}

\newcommand{\Z}{{\mathbb Z}}

\newcommand{\fg}{{\mathfrak g}}
\newcommand{\fsl}{{\mathfrak s}{\mathfrak l}}

\newcommand{\fp}{{\mathfrak p}}

\def\Pic{\mathop{\rm Pic}\nolimits}

\title[Rigidity of Schubert varieties]{Homological rigidity and Schur rigidity of Schubert varieties in rational homogeneous spaces}

\author{Cong Ding and Qifeng Li}

\begin{document}

\tikzset{->-/.style={decoration={
				markings,
				mark=at position .5 with {\arrow{>}}},postaction={decorate}}}
	\tikzset{-<-/.style={decoration={
				markings,
				mark=at position .5 with {\arrow{<}}},postaction={decorate}}}
	\tikzset{cross/.style={cross out, draw=black, minimum size=15*(#1-\pgflinewidth), inner sep=0pt, outer sep=0pt},
		cross/.default={1pt}}

\begin{abstract} 
 A Schubert variety $X_0$ on a rational homogenous space $X=G/P$ is said to be homologically rigid, if any subvariety $Z$ on $X$ representing the same homology class with $X_0$ must satisfy $Z=g\cdot X_0$ for some $g\in{\rm Aut_0}(X)$. We say $X_0$ is Schur rigid, if furthermore any subvariety $Z$ on $X$ whose homology class is a multiple $r$ of that of $X_0$ must satisfy $Z=g_1\cdot X_0+\cdots+g_r\cdot X_0$ for some $g_1,\cdots ,g_r\in{\rm Aut_0}(X)$.
Homological rigidity and Schur rigidity of Schubert varieties in rational homogeneous spaces of Picard number one have been well studied in extensive literature. In this paper, we study both rigidity problems of Schubert varieties in rational homogeneous spaces of higher Picard numbers. We show that in the long root cases, including all cases when $G$ is of type $ADE$, smooth Schubert varieties have homological rigidity. Besides, we give the complete list of Schubert varieties of subdiagram type with/without homological rigidity. Furthermore, for a Schubert variety $X_0$ of subdiagram type, we show that it has Schur rigidity in long root cases unless $X_0$ admits a fiber bundle structure over the projective space. 

\end{abstract}

\maketitle


\section{Introduction}
In this paper, we are going to study rigidity problems of Schubert varieties in rational homogeneous spaces of \textit{higher Picard numbers}. Let $X=G/P$ be a rational homogeneous space, where $G$ is a complex connected semisimple algebraic group and $P$ is a parabolic subgroup.  A Schubert variety $X_0$ is the clousure of a $B$-orbit (which is an affine cell) in $X$, where $B\subset P$ is the Borel subgroup. The homology classes of Schubert varieties give additive basis of the homology group of $X$, which are called Schubert classes. A classical problem in algebraic geometry is to describe all effective cycles representing homology classes in the ray generated by the Schubert classes. Trivially, we know the sum of translates of $X_0$ by $\mathrm{Aut}_0(X)$ gives such effective cycles, but quite often there are more. For example, consider the homology class of the linear subspace $\mathbb{P}^1\subset \mathbb{P}^n$,  homology class of any projective curve in $\mathbb{P}^n$ is the multiple of the homology class of such $\mathbb{P}^1$. However, if we consider the rational homogeneous space $X$ to be non-linear, some rigidity phenomenon would happen. For example, it can be shown that for a sub-Grassmannian in a Grassmannian, being both non-linear, all effective cycles representing the multiple of the Schubert class come from the sum of group translates (cf. \cite{Bry01, Hong07}).

In general, for a Schubert variety $X_0$ on a rational homogeneous space $X=G/P$, the pair $(X_0, X)$, 
is said to be \textit{Schur rigid} if for any $r\in \mathbb{Z}^+$, any subvariety $Z \subset X$ having homology class $r[X_0]$ (the square bracket denotes the homology class), must be $g_1\cdot X_0+\cdots+g_r\cdot X_0 $ for some $g_1,\cdots, g_r\in \operatorname{Aut}_0(X)$. In particular, if this holds for $r=1$, the pair $(X_0, X)$ is called \textit{homologically rigid}. There are numerous literatures on this problem, in particular in the case when the Picard number is one. See \cite{Cos11, CR13, Cos14, Hong05, HM13, HM20, Hong07, HK19, Rob13, RT12} and the references therein.

By making use of the geometric theory of variety of minimal rational tangents (VMRT for short) developed by Hwang-Mok, Hong-Mok\cite{HM13}, Hong-Kwon\cite{HK19} obtained the homological rigidity of smooth Schubert varieties in rational homogeneous space of Picard number one with a few exceptions. And the exceptions occur only when $X$ is marked at a short simple root (see Section \ref{marked_Dynkin}) and $X_0$ is linear.

To state the results, we need several notations on parabolic subgroups and marked Dynkin diagrams. Let $G$ be a simple algebraic group, $B$ be a Borel subgroup, and $R$ be the set of simple roots. Each parabolic subgroup $P$ of $G$ containing $B$ can be written as $P=P_J$ for a unique subset $J$ of $R$ such that $P_\emptyset=B$ and $P_R=G$, see Section \ref{semisimple_group} for more details. A rational homogeneous space $X=G/P_J=G/P_{I^c}$ is represented by a marked Dynkin diagram, namely marking the nodes $I=R\backslash J$ in the Dynkin diagram of $G$. The latter is called the marked Dynkin diagram and written as $\Gamma_R(I)$, see Section \ref{marked_Dynkin} for more details.

\begin{theorem}[Theorem 1.1 in \cite{HM13}, Theorem 1.4 in \cite{HK19}]\label{HM_homological}
    Let  $X=G/P_{I^c}$ be a rational homogeneous space of Picard number one  (i.e., $I=\{\alpha\}$ is a single element set) and $X_0$ be a smooth Schubert variety. If either 
    \begin{enumerate}
        \item $\alpha$ is a long root; or
        \item $X_0$ is non-linear,
    \end{enumerate}
    then any subvariety of $X$ having the same homology class as $X_0$ is induced by the action of $\operatorname{Aut}_0(X)$.

\end{theorem}

In this paper we study the cases when $X$ is of higher Picard numbers. Our main result for homological rigidity is as follows.


\begin{theorem}\label{main_thm_simple_ver}
     Let $X=G/P_{I^c}$ be a rational homogeneous space and $X_0$ be a smooth Schubert variety in $X$. If all marked roots in $I$ are long roots, then $X_0$ is homologically rigid. In particular, if $G$ is of type $ADE$, all smooth Schubert varieties on $X$ are homologically rigid.
\end{theorem}

Note that the homogeneity of the Schubert variety is not required in Theorem \ref{main_thm_simple_ver}. In fact we can prove a stronger version by applying Proposition \ref{prop homology rigidity reduction}, which also covers some cases when there are short roots in $I$ (for example, see Corollary \ref{subdiagram_rigid_marked_root_outside}). 
From Theorem \ref{smooth_BP} (Theorem 3.3, Theorem 3.6 in \cite{RS16}), we know in higher Picard number cases, smooth Schubert varieties are  iterated fiber bundles of smooth Schubert varieties in rational homogeneous spaces of Picard number one (which will be called Billey-Postnikov decomposition, see Section \ref{BP_decom_Schubert}). We can prove that if there are no exceptional pair in the sense of Definition \ref{exceptional_pair} (which is slightly broader than being not homologically rigid, see 
Remark \ref{covering_pair}) in the fibers, then homological rigidity holds. Applying Theorem \ref{HM_homological} and Proposition \ref{prop homology rigidity reduction}, we can get Theorem \ref{main_thm_simple_ver} easily.


In \cite{HM13}, the proof of homological rigidity can be reduced to characterize the local deformation of the Schubert variety and apply geometric theory of VMRT's. There are difficulties in generalizing the method in \cite{HM13} to higher Picard number cases, including the non-homogeneity of the Schubert varieties, and the parallel transport techniques in VMRT theory with \textit{several} minimal rational components. Instead, we give a proof which uses the existence of Billey-Postnikov decomposition  (see Section \ref{BP_decom_Schubert}) and the induction on Picard numbers.

As an application, we determine completely whether homological rigidity holds for $(X_0, X)$ supposing that $X_0 \subset X$ is 
a Schubert variety associated with a Dynkin subdiagram $\Gamma_L$, see Section \ref{Preli_Schubert_subdiagram} for the conventions. Here we say $X_0$ is of subdiagram type. Throughout this paper, we follow the root conventions from Bourbaki's \textit{Lie Groups and Lie Algebras}
 \cite{Bour02}.


\begin{theorem}\label{homological_subdiagram}
    Let $X=G/P_{I^c}$ be a rational homogeneous space and $X_0=G_0/P_0$
be a Schubert variety associated with a Dynkin subdiagram $\Gamma_L$. Then homological rigidity holds except for the following cases.
\begin{enumerate}
        \item  $X=(F_4,\{\alpha_3\}), L=\{\alpha_2,\alpha_3\}$ or $\{\alpha_3\}$;
        \item $X=(F_4,\{\alpha_4\}),  L=\{\alpha_3,\alpha_4\}$ or $\{\alpha_4\}$;
        \item$X=(F_4, \{\alpha_1,\alpha_3\}), L=\{\alpha_3\}$; 
        \item $X=(F_4, \{\alpha_1,\alpha_4\}), L=\{\alpha_4\}$ or $\{\alpha_3, \alpha_4\}$; 
        \item $X=(F_4,\{\alpha_1,\alpha_3, \alpha_4\}), L=\{\alpha_3\}$;
     \item $X=(F_4,\{\alpha_3, \alpha_4\}), L=\{\alpha_3\}$ or $\{\alpha_2,\alpha_3\}$;
         \item $X=(C_n,\{\alpha_k\}\cup R_1)(n\geq 3), L=\{\alpha_{k},\alpha_{k+1}\cdots, \alpha_{b-1}\}\cup R_2$,
         $2\leq k< b\leq n$, $R_1\subset \{\alpha_1,\alpha_2,\dots, \alpha_{k-2}, \alpha_{k-1}\}, R_2\subset \{\alpha_1,\alpha_2,\dots, \alpha_{k-2}\}$ ($R_2=\emptyset$ if $k=2$);
         \item $X=(B_n, \{\alpha_n\} \cup R_1)(n\geq 3), L=\{\alpha_{b+1},\alpha_{b+2}, \cdots, \alpha_n\}\cup R_2$,
         $1<  b <n$, $R_1 \subset \{\alpha_1, \alpha_2, \dots, \alpha_{b-1}\}, R_2\subset \{\alpha_1, \alpha_2, \dots, \alpha_{b-2}\}$($R_2=\emptyset$ if $b=2$), $R_1\neq \emptyset$.
    \end{enumerate}
Here, $(G,I)$ denotes the pair of algebraic group and the set of marked roots determining $G/P_{I^c}$. 
Moreover, the cases listed above are not homologically rigid.
\end{theorem}

\begin{remark}
In the literature such as \cite{HM13}, the authors use $\Lambda$ to denote the set of simple roots in $\Gamma\backslash \Gamma_L$ which are adjacent to the subdiagram $\Gamma_L$, and in the Picard number one case, the subdiagram can be uniquely determined by $\Lambda$. 
    However, in higher Picard number case, in general $\Lambda$ cannot determine the subdiagram uniquely, since there may exist an additional connected component lying at the leftmost (or rightmost) position of the subdiagram, where the set of simple roots adjacent to the subdiagram remains the same. 
    We thus adopt the terminology in Theorem \ref{homological_subdiagram} in an alternative manner.
\end{remark}

\begin{corollary}\label{cor_marked_roots_contained_in_L}
   In the setting of Theorem \ref{homological_subdiagram}, the following hold.
   \begin{enumerate}
       \item[(i)]  If $I\subset L$,  then the homological rigidity holds except for Cases (1),(2), and Cases (7),(8) with $R_1\subset R_2$;
       \item[(ii)]  If $L$ is connected, then the homological rigidity holds except for Cases (1)-(6) and Cases (7),(8) with $R_2=\emptyset$.
   \end{enumerate}
\end{corollary}
   From a private communication, we know that J. Hong also proves the conclusion (i) of Corollary \ref{cor_marked_roots_contained_in_L}, where she uses the cohomology method.

In \cite{HM20}, Hong-Mok obtained the following result on Schur rigidity.

\begin{theorem}[Theorem 1.1 in \cite{HM20}]\label{Schur_HM}
    Let $X = G/P$ be a rational homogeneous space of Picard number
one and let $X_0$ be a non-linear smooth Schubert variety of $X$. Then $X_0$ has Schur rigidity in $X$.
\end{theorem}
In this paper, we obtain the following result on Schur rigidity for the subdiagram case. 
\begin{theorem}\label{Schur_subdiagram}
 Let $X=G/P_{I^c}$ be a rational homogeneous space, and $X_0$ be a Schubert variety associated with a Dynkin subdiagram $\Gamma_L$. Suppose that all marked roots in $I\cap L$ are long roots, and $X_0$ is not a fiber bundle over the projective space. Then $X_0$ has Schur rigidity in $X$.
\end{theorem}

\begin{remark}
  Note that if the marked roots in $I\cap L$ are allowed to be short roots, non-rigid cases would appear even if we assume other conditions, see for example Case (8) in Theorem \ref{homological_subdiagram}. The condition that $X_0$ is not a fiber bundle over the projective space is natural, from Theorem \ref{Schur_HM}. 
\end{remark}

Finally, we note that, in a different manner, similar problems have been considered by Liu-Sheshmani-Yau in \cite{LSYau24, LSY24} for types $ABD$.

The rest of this paper will be organized as follows. In Section \ref{s.preliminaries}, we give some preliminaries, where we fix some conventions for the forthcoming proof. In Section \ref{homological_higher_picard}, we give the proof of our results on homological rigidity (Theorem \ref{main_thm_simple_ver}, \ref{homological_subdiagram}). In Section \ref{Schur_higher_Pic}, we give the proof of our results on Schur rigidity (Theorem \ref{Schur_subdiagram}).

\subsection*{Acknowledgements}
The authors would like to thank Jaehyun Hong and Ngaiming Mok for helpful discussions.
Ding is supported by a start-up funding from Shenzhen University and Shenzhen Peacock Plan, Li is supported by the Natural Science Foundation of Shandong Province (Grant No. ZR2025QA08) and by the National Natural Science Foundation of China (Grant No. 12571043) .

\section{Preliminaries}\label{s.preliminaries}

\subsection{Semisimple algebraic groups and Lie algebras}\label{semisimple_group}

Let $G$ be a complex connected semisimple algebraic group of rank $r$, $B$ be a Borel subgroup and $T<B$ be a maximal torus. Let $\Sigma$ (resp. $\Sigma^+$, $\Sigma^-$) be the set of roots (resp. positive roots, negative roots) of $G$, and $R\subset\Sigma$ be the set of simple roots, written as $R=\{\alpha_1,\ldots,\alpha_r\}$. For each simple root $\alpha$, let $P_\alpha<G$ be the corresponding minimal parabolic subgroup containing $B$. Given a subset $A$ of $R$, denote by $P_A<G$ the parabolic subgroup generated by those $P_\alpha$ with $\alpha\in A$. Let $P_A^-<G$ be the opposite parabolic subgroup of $P_A$, and $L_A$ be the standard Levi subgroup of them, i.e., $L_A=P_A\cap P_A^-$. In particular $P_\emptyset=B$, $L_\emptyset=T$, and $P_R=L_R=G$.

Let $\mathfrak{t}\subset\mathfrak{b}\subset\mathfrak{g}$ be the Lie algebras of $T<B<G$. Given a subset $A\subset R$, we use $\mathfrak{p}_A$, $\mathfrak{p}^-_A$ and $\mathfrak{l}_A$ to represent the Lie algebras of $P_A$, $P_A^-$ and $L_A$ respectively. Take a root $\alpha$ of $G$. Let $\mathfrak{g}_\alpha$ and $\mathfrak{g}_{-\alpha}$ be the root spaces corresponding to roots $\alpha$ and $-\alpha$ respectively. Note that $[\mathfrak{g}_\alpha, \mathfrak{g}_{-\alpha}]$ is a 1-dimensional subspace of the Cartan subalgebra $\mathfrak{t}$, and we denote it by $\mathfrak{t}_\alpha$ or  $\mathfrak{t}_{-\alpha}$. Now $\mathfrak{g}_\alpha\oplus\mathfrak{t}_\alpha\oplus\mathfrak{g}_{-\alpha}$ is a subalgebra of $\mathfrak{g}$ that is isomorphic to $\fsl_2$, and we denot it by $\fsl_2(\alpha)$. Let $U_\alpha$, $U_{-\alpha}$,  $T_\alpha$ and $G_\alpha$ be the connected algebraic subgroups of $G$ whose Lie algebras are $\mathfrak{g}_\alpha$, $\mathfrak{g}_{-\alpha}$, $\mathfrak{t}_\alpha$ and $\fsl_2(\alpha)$ respectively. Given a subset $A$ of $R$, denote by $U_A$ (resp. $U_A^-$, $T_A$, $G_A$) the connected algebraic subgroups generated by those $U_\alpha$ (resp. $U_{-\alpha}$, $T_\alpha$, $G_\alpha$) with $\alpha\in A$. In particular, $U_A$ and $U_A^-$ are unipotent groups, $T_A$ is a torus, and $G_A$ is the semisimple part of $L_A$.  If $A'$ is another subset of $R$, then $P_{A'}\cap G_A$ (resp. $P_{A'}\cap L_A$) is a parabolic subgroup associated with the subset $A'\cap A$ of the set of simple roots $A$ of $G_A$ (resp. of $L_A$). When there is no danger of confusion, we also use $P_{A'}$ to represent the parabolic $P_{A'}\cap G_A$ (resp. $P_{A'}\cap L_A$) of $G_A$ (resp. $L_A$). For example, by $G_A/P_{A'}$ or $L_A/P_{A'}$ we mean the rational homogeneous space $G_A/(P_{A'}\cap G_A)=L_A/(P_{A'}\cap L_A)$.

\subsection{Schubert varieties}
Let $W$ and $W_P$ be  be the Weyl groups of $G$ and reductive part of $P$ respectively. Then the set of $T$-fixed points in $X = G/P$ is indexed by the set of right cosets $W/W_P$. Let $W^P\cong W/W_P$ be the subset of minimal length representatives of $W/W_P$ in $W$. Let $o$ be the base point of $X$ with respect to $P$, then $w \leftrightarrow x_w:=w.o$ gives the bijection between $W^P$ and the set of $T$-fixed points. The $B$-orbit decomposition of $X$ is given by \[X= \bigsqcup_{w\in W^P}B.x_w\]
Then the closure $\overline{B.x_w}$ is called the Schubert variety of type $w$.

The following propositions are proved in \cite{HM13}, where the Picard number condition is not necessary.

\begin{proposition}[Proposition 2.1 in \cite{HM13}]\label{Schubert closed}
     Let $X = G/P$ be a rational homogeneous space. Let $Z$ be an irreducible closed subvariety of $X$. Consider $Z$ as a point $[Z]$ in the Chow variety $\operatorname{Chow}(X)$ of $X$. Then the $G$-orbit of $[Z]$ is closed if and only if $Z$ is a Schubert variety.
\end{proposition}

The homological rigidity can be reduced to characterize the local deformation of Schubert varieties. More precisely, one has
\begin{proposition}[Proposition 2.2 in \cite{HM13}] \label{reduced_to_local_deform}
    Let $X = G/P$ be a rational homogeneous space and let $X_0$ be a smooth Schubert variety. Suppose that any local deformation of $X_0$ in $X$ is induced by the action of G. Then any subvariety of $X$ having the same homology class as $X_0$ is induced by the action of $G$.
\end{proposition}

\subsection{Marked Dynkin diagrams}\label{marked_Dynkin}

Let $\Gamma$ be the Dynkin diagram of $G$, whose nodes are denoted by the corresponding simple roots $\alpha_1,\ldots,\alpha_r$ respectively. Given a subset $A$ of $R$, denote by $\Gamma_A$ the subdiagram of $\Gamma$ such that the nodes are exact $A$, and an edge of $\Gamma$ lies in $\Gamma_A$ if and only if both ends of this edge are contained in $A$. In particular, $\Gamma_A$ is the Dynkin diagram of $G_A$.

Rational homogeneous spaces can be represented by \emph{marked Dynkin diagrams} in the following way. A rational homogeneous space $X$ under the action of $G$ can be written as $X=G/P_{I^c}$ for a subset $I$ of $R$, where $I^c$ represents the complement $R\setminus I$ of $I$. We mark the nodes $I$ in the Dynkin diagram $\Gamma$, denoted as $\Gamma(I)$ or $\Gamma_R(I)$ for precision. The correspondence has such meaning: each root $\alpha\in I$ corresponds to a $B$-stable prime divisor $D_\alpha$ on $S$, and the Picard group $\Pic(X)$ of $X$ is freely generated by $\mathcal{O}_X(D_\alpha)$. In particular, we can write $\Pic(X)=\oplus_{\alpha\in I}\mathbb{Z}[D_\alpha]$. 

For convenience, we sometimes write $G/P_{I^c}$ as $(G, I)$.


\subsection{Schubert varieties associated with a subdiagram}\label{Preli_Schubert_subdiagram}

Let $L$ be a subset of $R$. Denote by $\Gamma_L(I\cap L)$ or $\Gamma_L(I)$ the Dynkin diagram $\Gamma_L$ with marked nodes $I\cap L$. We call it \emph{a subdiagram of the marked Dynkin diagram} $\Gamma(I)$. The rational homogeneous space $G_L/(P_{I^c}\cap G_L)$ is a smooth Schubert variety of $X$, and it is represented by the subdiagram $\Gamma_L(I\cap L)$ of the marked Dynkin diagram $\Gamma(I)$. A typical example of smooth Schubert varieties associated with a subdiagram is the sub-Grassmannians in Grassmannians.

A subset $L$ of $R$ is said to be \emph{reduced} with respect to $X=G/P_{I^c}$, if there is no proper subset $K$ of $L$ such that the two Schubert varieties $G_K/(G_K\cap P_{I^c})$ and $G_{L}/(G_{L}\cap P_{I^c})$ coincide with each other. In other words, a subset $L$ of $R$ is said to be \emph{nonreduced} with respect to $X=G/P_{I^c}$, if the marked Dynkin diagram $\Gamma_{L}(I)$ has a connected component such that none of its nodes is contained in $I$, the set of marked nodes.


By the \emph{neighbor} of a subset $A$ of $R$, we mean the set $N(A)$ of simple roots outside $\Gamma_A$ which are adjacent to $\Gamma_A$ in the Dynkin diagram $\Gamma_R$, i.e.,
\begin{eqnarray*}
N(A):=\{\alpha\in A^c\mid \mbox{ the Cartan pairing } \langle\alpha, \beta\rangle<0 \mbox{ for some } \beta\in A\}.
\end{eqnarray*}
When the set $A$ consists of a single element, saying $\alpha\in I$, we simply write $N(\alpha)$  to represent $N(\{\alpha\})$.

\begin{lemma}\label{stab.subdiagram}
Let $X=G/P_{I^c}$ be a rational homogeneous space, and $X_0=G_L/(P_{I^c}\cap G_L)$ be a Schubert variety associated with a reduced subset $L$. Then the stabilizer of $X_0$ is the parabolic subgroup $P_{(N(L)\cup I\backslash L)^c}$ of $G$.
\end{lemma}

\begin{proof}
First, let us verify that $P_{(N(L)\cup I\backslash L)^c}$ stabilizes $X_0$.
   This can be obtained by the following Tits fibration \cite{Tits67}.
   \begin{eqnarray}\label{diagram-Tits}
\xymatrix{& G/P_{(I\cup N(L))^c}\ar[ld]_{\pi_1}\ar[rd]^{\pi_2}  & \\
G/P_{I^c}  & &G/P_{(N(L)\cup I\backslash L)^c}  \\}
\end{eqnarray}
We know the fiber of $\pi_2$ is isomorphic to 

\begin{eqnarray}\label{eqn-stabilizer}
&& P_{(N(L)\cup I\backslash L)^c}/P_{(I\cup N(L))^c} \\
 &\cong& G_{(N(L)\cup I\backslash L)^c}/ G_{(N(L)\cup I\backslash L)^c}\cap P_{(I\cup N(L))^c} \nonumber\\
     &=& G_{N(L)^c\cap (I\backslash L)^c}/G_{N(L)^c\cap (I\backslash L)^c}\cap P_{I^c}\cap P_{N(L)^c}\nonumber\\
     &=& G_{N(L)^c\cap (I\backslash L)^c}/G_{N(L)^c\cap (I\backslash L)^c}\cap P_{I^c} \nonumber\\ 
 &=& G_L/(P_{I^c}\cap G_L) \nonumber
\end{eqnarray}
 Here we explain the last equality in formula \eqref{eqn-stabilizer}. Note that $L\subset N(L)^c$, we have
\begin{eqnarray*}\label{eqn-N(L)IJ}
\quad\quad\quad N(L)^c\cap (I\backslash L)^c=(N(L)^c\cap I^c)\cup (N(L)^c\cap L)=(N(L)^c\cap I^c)\cup L.
\end{eqnarray*} 
Then we can conclude that
\begin{itemize}
\item[(i)] every connected component of $L$ is also a connected component of $N(L)^c\cap (I\backslash L)^c$;
\item[(ii)] given $\alpha\in N(L)^c\cap (I\backslash L)^c$, we have $\alpha\in I$, i.e., $\alpha$ is a marked root, if and only if $\alpha\in L$.
\end{itemize}
The last equality in \eqref{eqn-stabilizer} follows from assertions (i) and (ii) immediately.

Hence for any $x\in G/P_{(N(L)\cup I\backslash J)^c}$, $\pi^{-1}_2(x)$ is isomorphic to $X_0$, $\pi_1(\pi^{-1}_2(x))$ is a $G$-translate of $X_0\subset X$ and $P_{(N(L)\cup I\backslash J)^c}$ stabilizes $X_0\subset X$.
   
   Let $H<G$ be the stabilizer of $X_0\subset X$. We have shown that $H$ contains the parabolic subgroup $P_{(N(L)\cup I\backslash L)^c}$. Then there exists a subset $A$ of $R$ such that $H=P_A$ and $A\supset (N(L)\cup I\backslash L)^c$. To complete the proof, it remains to show that $N(L)\cup I\backslash L\subset A^c$.

Consider the natural projection $\pi: G\rightarrow X=G/P_{I^c}$ and the tangent map $d_e\pi: \mathfrak{g}=T_eG\rightarrow T_oX$, where $o:=P_{I^c}/P_{I^c}$ is the base point of $X$. Then the inverse image $(d_e\pi)^{-1}(T_oX_0)=\mathfrak{g}_L+\mathfrak{p}_{I^c}$, where $\mathfrak{g}_L$ and $\mathfrak{p}_{I^c}$ are the Lie algebras of $G_L$ and $P_{I^c}$ respectively. Hence for any  $\alpha\in I\backslash L$, we have  $T_eU_{-\alpha}=\mathfrak{g}_{-\alpha}\not\subset (d\pi_e)^{-1}(T_oX_0)$ and thus $T_o(U_{-\alpha}\cdot o)=d\pi_e(T_eU_{-\alpha})\not\subset T_oX_0$. It follows that $U_{-\alpha}\not\subset H=P_A$ and $\alpha\notin A$, implying $I\backslash L\subset A^c$. 

 Take any $\beta\in N(L)\cap I^c$. Then there exists a connected component $\Gamma_K$ of $\Gamma_L$ containing a node adjacent to $\beta$ in the Dynkin diagram of $G$. Let $\gamma$ be the highest root of $G_K$. Then the Cartan pairing $\langle\beta, \gamma\rangle<0$ and thus $\pm(\beta+\gamma)$ are roots of $G$. In particular, $[\fg_{-\beta}, \fg_{-\gamma}]=\fg_{-\beta-\gamma}\neq 0$. Since $L$ is reduced with respect to $X$, we have $\fg_{-\gamma}\subset\fg_L\setminus\fp_{I^c}$. In particular, 
\begin{eqnarray*}
\gamma\in\sum_{\zeta\in L}\Z \zeta \mbox{ and } \gamma\not\in\sum_{\zeta\in I^c}\Z \zeta.
\end{eqnarray*}
Recall that $\beta\in N(L)\cap I^c\subset L^c\cap I^c$. Then we have
\begin{eqnarray*}
\beta+\gamma\not\in \sum_{\zeta\in L}\Z \zeta \mbox{ and } \beta+\gamma\not\in\sum_{\zeta\in I^c}\Z \zeta.
\end{eqnarray*}
Then we can conclude $\fg_{-\beta-\gamma}\nsubseteq\fg_L+\fp_{I^c}$ by considering the weights of these $T$-modules.

Now the $U_{-\beta}$-action on $X$ fixes the base point $o\in X$, 
and the induced $U_{-\beta}$-action on $T_oX$ does not stabilize the subspace $T_oX_0$. Hence $U_{-\beta}$ does not stabilize the submanifold $X_0$ of $X$ and thus $N(L)\cap I^c\subset A^c$. As a consequence, $N(L)\cup I\backslash L=(N(L)\cap I^c)\cup (I\backslash L)\subset A^c$. 
This completes the proof.
\end{proof}

A Schubert variety of subdiagram type can be written as $G_L/(G_L\cap P_{I^c})$ for some $L\subset R$, while the choice of $L$ may not be unqiue. 

For a subset $L$ of $R$, we can define $L_{\rm red}$ to be its \emph{reduction} with respect to $X=G/P_{I^c}$, i.e., $L_{\rm red}$ is reduced and $G_L/(G_L\cap P_{I^c})=G_{L_{\rm red}}/(G_{L_{\rm red}}\cap P_{I^c})$. Define $L_{\rm pro}$ to be $(N(L_{\rm red})\cap I\backslash L_{\rm red})^c$, and we call it the \emph{prolongation} of $L$ (hence also the prolongation of $L_{\rm red}$).

\begin{corollary}
Take two subsets $L, L'$ of $R$. Then the Schubert varieties $G_{L'}/(G_{L'}/\cap P_{I^c})=G_L/(G_L/\cap P_{I^c})$ if and only if $L_{\rm red}\subset L'\subset L_{\rm pro}$.
\end{corollary}

\begin{proof}
Denote by $S:=G_L/(G_L/\cap P_{I^c})$.
Suppose $G_{L'}/(G_{L'}/\cap P_{I^c})=S$. By definition of reduction, we have $L_{\rm red}=L'_{\rm red}\subset L\cap L'$. Since $P_{L'}$ stabilizes the Schubert variety $S$ on $X$, then $L'\subset L_{\rm pro}$ by Proposition \ref{stab.subdiagram}. Hence $L_{\rm red}\subset L'\subset L_{\rm pro}$.

Conversely, suppose $L_{\rm red}\subset L'\subset L_{\rm pro}$.  By Proposition \ref{stab.subdiagram} the parabolic subgroup  $P_{L'}$ stabilizes $S\subset X$. In particular, $P_{L'}\cdot  o\subset S$, where $o$ is the base point of $X=G/P_{I^c}$. Note that $G_{L'}/(G_{L'}\cap P_{I^c})=P_{L'}\cdot o\supset P_{L_{\rm red}}\cdot o=S$. It follows that $G_{L'}/(G_{L'}\cap P_{I^c})=S$.
\end{proof}

\subsection{Mori contraction on rational homogeneous spaces}
\label{Mori_contraction_prelimi}
Each Mori contraction on the rational homogenous space $X=G/P_{I^c}$ is of the form $\pi_A: G/P_{I^c}\rightarrow G/P_{I^c\cup A}$ for some subset $A\subset I$. Conversely, for every subset $A$ of $I$, the morphism $\pi_A$ is a Mori fibration. The central fiber of $\pi_A$ is the rational homogeneous space $X^A:=P_{I^c\cup A}/P_{I^c}=G_{I^c\cup A}/(P_{I^c}\cap G_{I^c\cup A})$ associated with the marked subdiagram $\Gamma_{I^c\cup A}(A)$. The morphism $\pi_A$ is nothing else but the natural projection map of principal bundle $G/P_{I^c}=G\times^{P_{I^c\cup A}} X^A\rightarrow G/P_{I^c\cup A}$. 

In terms of marked Dynkin diagrams, we give the following example, where $G=A_5, I=\{\alpha, \beta\}$, the central fibers of the projection maps $\pi_\alpha, \pi_\beta$ correspond to the subdiagrams in the dotted and dashed boxes respectively.
    \begin{center}
  \begin{tikzpicture}[scale=0.5]
  \draw[thick] (0,0) -- (6, 0)  ; 
\draw[ thick, fill=white] (0,0) circle (3pt) node[above, outer sep=3pt]{};
		\draw[ thick, fill=black] (1.5,0) circle (3pt) node[above, outer sep=3pt]{$\beta$};
        \draw[ thick, fill=white] (3,0) circle (3pt) node[above, outer sep=3pt]{};
        \draw[ thick, fill=black] (4.5,0) circle (3pt) node[above, outer sep=3pt]{$\alpha$};
         \draw[ thick, fill=white] (6,0) circle (3pt) node[above, outer sep=3pt]{};

          \draw[thick,dashed] (-0.4,0.7) -- (3.4,0.7)  ;
          \draw[thick,dashed] (-0.4,-0.7) -- (3.4,-0.7)  ;
           \draw[thick,dashed] (-0.4,0.7) -- (-0.4,-0.7)  ;
             \draw[thick,dashed] (3.4,0.7) -- (3.4,-0.7)  ;

        \draw[thick,dotted] (2.6,0.9) -- (6.4,0.9)  ;
          \draw[thick,dotted] (2.6,-0.9) -- (6.4,-0.9)  ;
           \draw[thick,dotted] (2.6,0.9) -- (2.6,-0.9)  ;
             \draw[thick,dotted] (6.4,0.9) -- (6.4,-0.9)  ;
	\end{tikzpicture} 	
   \end{center} 

\subsection{The Billey-Postnikov decomposition of a smooth Schubert variety}\label{BP_decom_Schubert}
When $X=G/P_{I^c}$ is of Picard number one, we know that all smooth Schubert varieties are either of subdiagram type or some non-homogeneous horospherical varieties by suitable embeddings, where the latter only appears when the marked root is short (cf. \cite[Proposition 3.7]{HM13}, \cite[Theorem 1.2, Theorem 1.3]{HK19}). In higher Picard number case, by the result of Rchmond and Slofstra \cite[Theorem 3.3, Theorem 3.6, Corollary 3.7]{RS16}, a smooth Schubert variety can be described as an iterated fiber bundle such that the central fiber of each step is a smooth Schubert variety in rational homogeneous spaces of Picard number one (which are called Grassmannian Schubert varieties in \cite{RS16}).

More precisely, a smooth Schubert variety has an iterated fiber bundle structure such that each step is a subbundle of $G/P_{I^c\cup\{\beta_1,\ldots,\beta_{t-1}\}}\rightarrow G/P_{I^c\cup\{\beta_1,\ldots,\beta_t\}}$ for some $\beta_1,\ldots,\beta_t\in I$ over a smooth Schubert variety on the base $G/P_{I^c\cup\{\beta_1,\ldots,\beta_t\}}$. The choice of these $\beta_j$ depends on some combinatorial information related to the Schubert variety.

 Let $J=I^c$ and $X=G/P_J$. Let $W$ and $W_J$ be the Weyl groups of $G$ and reductive part of $P_J$ respectively. Then Schubert varieties on $X$ can be indexed by $W/W_J$. Let $W^J \cong W/W_J$ be the subset of minimal length representatives of $W/W_J$ in $W$. The Schubert variety associated to $w\in W^J$ is written as 
\[
\begin{aligned}
    X^J(w):&=\overline{BwP_J/P_J} \\
    &=\bigsqcup_{w'\leq w, w'\in W^J} Bw'P_J/P_J,
\end{aligned}
\]
where $\leq $ denotes the Bruhat order on $W$.

If $J \subset K \subset R$, then every element can be written uniquely as $w=vu$ where $v\in W^K$ and $u\in W_K \cap W^J$. This is called the (right) parabolic decomposition of $w$ with respect to $K$. The Poincar\'{e} polynomial $P^{J}_w(t)$ of an element $w\in W^J$ is defined as
\[ P_w^J(t):=\sum_{w'\leq w, w'\in W^J}t^{\ell(w')}\]
where $\ell(w')$ denotes the length of $w'$.

\begin{definition}
Let $w\in W^J$ and $w=vu$ be a parabolic decomposition with respect to $K$, where $J\subset K\subset R$. We say $w=vu$ is a Billey-Postnikov decomposition with respect to $(J,K)$ if 
\[P^J_w(t)=P^K_v(t)\cdot P_u^{J}(t).\]

\end{definition}
In \cite{RS16}, it is proven that
\begin{theorem}[from Theorem 3.3, Theorem 3.6 in \cite{RS16}]\label{smooth_BP}
    Let $w\in W^J$ and $w=vu$ be a parabolic decomposition with respect to $K$. Then under the Mori contraction $\pi_{K\backslash J}$, $\pi_{K\backslash J}: X^{J}(w)\rightarrow X^K(v)$ is Zariski-locally trivial with fiber $X^J(u)$ if and only if $w=vu$ is a Billey-Postnikov decomposition with respect to $(J,K)$. Moreover if $X^J(w)$ is smooth and its Picard number is at least 2, then there always exists a simple root $\alpha$ and a Billey-Postnikov decomposition with respect to $(J,K)$ where $K=J\cup \{\alpha\}$.
\end{theorem}
\begin{remark}
    Note that in the subdiagram case, any choice of $K\supset J$ would give a Billey-Postnikov decomposition with respect to $(J,K)$.
\end{remark}
From Theorem \ref{smooth_BP} one can see that a smooth Schubert variety $X^{J}(w)$ can be written as iterated fiber bundles of the following form
\[X^{J}(w)=X_0 \rightarrow X_1 \rightarrow \cdots \rightarrow X_{m-1}\rightarrow X_m=\operatorname{pt}\]
where each morphism is a Zariski-locally trivial fiber bundle with respect to some Billey-Postnikov decomposition and the fibers are smooth Schubert varieties in some rational homogeneous spaces of Picard number one. For the convenience of discussions, such a smooth Schubert variety in a rational homogeneous space of Picard number one is called a \textit{pair}. Also, the iterated fiber bundle structure is called the Billey-Postnikov decomposition of a smooth Schubert variety.

\begin{example}
Let $X=G/B$ be the full flag variety of $\mathbb{C}^3$, where $G={\rm PGL}(3)$ and $B$ is a Borel subgroup of $G$. Let $E_1\subset E_2 \subset \mathbb{C}^3$ the the standard flag and
\[X^{\emptyset}(w)=\{ F^{\bullet}: F_2\supset E_1 \} \cong \mathbb{F}_1.\]
 The two Mori contractions of $G/B$ are \[\pi_1 : X=G/B\rightarrow G/P_1=\{F_1: F_1\subset \mathbb{C}^3\}\cong \mathbb{P}^2\] and \[\pi_2 : X=G/B\rightarrow G/P_2=\{F_2: F_2\subset \mathbb{C}^3\}\cong \mathbb{P}^2.\]
 Under $\pi_1$ we have $\pi_1(X^{\emptyset}(w))=X^1(v)=G/P_1$, and \[\pi_1^{-1}(F_1)=\begin{cases}
     \{F^{\bullet}: F_1=E_1\}\cong \mathbb{P}^1,  (F_1= E_1)\\
     \{F^\bullet: F_2=F_1+E_1\}\cong \operatorname{pt}, (F_1\neq E_1)
 \end{cases};\] 
  Under $\pi_2$ we have $\pi_2(X^{\emptyset}(w))=X^2(v')=\{F_2: F_2\supset E_1\}\cong \mathbb{P}^1$, and \[\pi_2^{-1}(F_2)=\{F_1: F_1\subset F_2\}\cong \mathbb{P}^1.\]

Hence $\pi_1 : X^{\emptyset}(w)\rightarrow X^1(v)=G/P_1$ is not a Zariski-locally trivial fiber bundle but $\pi_2: X^{\emptyset}(w)\rightarrow X^2(v')$ is.

  If we look at the Schubert indices, we know $w=s_1s_2, v=s_1s_2, u=\operatorname{id}$. Then we have $P^{\emptyset}_w(t)=t^2+2t+1$ and $P_v^1(t)=t^2+t+1, P^\emptyset_u(t)=1$ and hence $w=vu$ is not a Billey-Postnikov decomposition. On the other hand, $v'=s_1, u'=s_2$, we have $P_{v'}^2(t)=t+1, P^\emptyset_{u'}(t)=t+1$ and hence $w=v'u'$ is a Billey-Postnikov decomposition.
\end{example}

\section{Homological rigidity in higher Picard number}\label{homological_higher_picard}

\subsection{The cases of Picard number one}

From Proposition \ref{reduced_to_local_deform}, homological rigidity can be reduced to the characterization of local deformation of the smooth Schubert variety.
\begin{definition}\label{exceptional_pair}
    Let $J=I^c$ and $X=G/P_J$ be a rational homogeneous space of Picard number one and $X_0 \subset X=G/P_J$ be a smooth Schubert variety. We say $(X_0, X)$ is an exceptional pair if there exists some local deformation of $X_0$ which is not obtained from the group action of $G$.
\end{definition}
\begin{remark}\label{covering_pair}
    If $(X_0,X)$ is not homologically rigid, then by Proposition \ref{reduced_to_local_deform}, it is an exceptional pair. Conversely, an exceptional pair $(X_0, X)$ is not homologically rigid, unless $X=G/P$ is isomorphic to $G'/P'$ such that $G'={\rm Aut_0}(X)$ contains $G$ as a proper subgroup. All these cases can be listed as follows.

\begin{lemma}[cf. \cite{HM13}]\label{lem-aut-larger}   
A rational homogeneous space $X=G/P$ is isomorphic to $G'/P'$ such that $G'={\rm Aut_0}(X)$ contains $G$ as a proper subgroup if and only  if one of the following holds:
    \begin{enumerate}
        \item $X=(B_n, \{\alpha_n\})\cong (D_{n+1}, \{\alpha_n\})$; 
        \item $X=(C_n, \{\alpha_1\}) \cong (A_{2n-1}, \{\alpha_1\})$;
        \item $X=(G_2, \{\alpha_1\}) \cong (B_3, \{\alpha_1\})$.
    \end{enumerate}
\end{lemma}

    In these cases $X$ is isomorphic to some $G'/P'$ which is marked at a long root, hence  by Theorem \ref{HM_homological} homological rigidity still holds. We define the exceptional pair in a broader sense, for the reason that when we discuss the higher Picard number case,  such isomorphism of rational homogeneous spaces is no longer appearing, and we need the group actions to be induced from $G$ in each step of the iterated fiber bundle associated with the Schubert variety in the proof of Proposition \ref{prop homology rigidity reduction}. 
\end{remark}

From \cite[Theorem 1.1, Theorem 1.2]{HM13}, \cite[Theorem 1.4]{HK19}, all other exceptional pairs come from the cases when $X=G/P_J$ is associated to a short root and $X_0$ is a linear subspace. In particular, we have

\begin{proposition}\label{prop-exceptional-pair}
Let $(X_0, X)$ be an exceptional pair associated with a subdiagram $\Gamma_L(I\cap L)$ of the marked Dynkin diagram $\Gamma(I)$. Then one of the following holds:
     \begin{enumerate}
    \item $X$ is in the list of Lemma \ref{lem-aut-larger} and $0<\dim X_0<\dim X$;
        \item $X=(C_n,\{\alpha_k\}), L=\{\alpha_{k},\alpha_{k+1},\cdots \alpha_{b-1}\}, 2\leq k< b\leq n$;
        \item  $X=(F_4,\{\alpha_3\}), L=\{\alpha_{2}, \alpha_3\}$ or $\{\alpha_3\}$;
        \item $X=(F_4,\{\alpha_4\}), L=\{\alpha_3,\alpha_4\}$ or $\{\alpha_4\}$,
    \end{enumerate}
         where $(G,I)$ denotes the pair of algebraic group and the set of marked roots determining $G/P_{I^c}$.
\end{proposition}

Note that all the cases (2)(3) and (4) in Proposition \ref{prop-exceptional-pair} are exceptional pairs, while whether the cases in Proposition \ref{prop-exceptional-pair} (1) are exceptional pairs depends on the choice of the algebraic group $G$ associated with $X$.

   For later convenience, we also provide the following lemmas.
\begin{lemma}\label{C_n_exceptional}
    Let $X=(G,I)=(C_n,\{\alpha_k\}), 1\leq k\leq n-1$, and $L=\{\alpha_{k},\alpha_{k+1},\cdots \alpha_{n}\}$.
    Let $Y$ be the homogeneous subspace associated with the subdiagram $\Gamma_L(\alpha_k)$, which is a maximal linear subspace in $X$. Then $(X_0, X)$ is an exceptional pair \textbf{(not necessarily associated with a subdiagram)}, if and only if $X_0\subsetneq g\cdot Y$ is a linear subspace for some $g\in G$. 
\end{lemma}
\begin{proof}
    When $k=1$, it follows from Lemma \ref{lem-aut-larger}. When $k\geq 2$, exceptional pairs appear only if $X_0$ is linear (cf. Theorem \ref{HM_homological}, Lemma \ref{lem-aut-larger}), and it must be contained in a maximal linear subspace. There are two types of maximal linear subspaces in $X$ (cf. \cite[Remark 5.7]{LM03}). According to \cite[Theorem 1.1]{HM13}, $(X_0, X)$ is an exceptional pair only if $X_0\subsetneq g\cdot Y$ for some $g\in G$. Conversely, if $X_0\subsetneq g\cdot Y$ for some $g\in G$, since $\mathrm{Aut}_0(Y)=A_{2n-2k+1}$ and one can always find a deformation from a group element $g\in A_{2n-2k+1}$ that does not lie in $C_n$, $(X_0, X)$ is an exceptional pair.
\end{proof}

\begin{lemma}\label{exceptional_triple}
    Let $X$ be a rational homogeneous space and $X_0$ be a homogeneous subspace associated with a subdiagram. Suppose that $X_1\subset X_0 \subsetneq X$ is a smooth Schubert variety in $X$ (and hence also a smooth Schubert variety in $X_0$). Suppose further that $X_1,X_0$ are of Picard number one and $(X_1,X_0)$ is an exceptional pair. Then we have 
    \begin{enumerate}
        \item $(X_1,X)$ is also an exceptional pair when $X$ is of Picard number one;
        \item $X_1$ is not homologically rigid in $X$ when $X$ is of Picard number at least two.
    \end{enumerate}
\end{lemma}
\begin{proof}
$(X_1,X_0)$ is an exceptional pair only if $X_0$ is of types $B_n, C_n,F_4,G_2$. Since the Dynkin diagrams of $F_4$ and $G_2$ could not be extended, it remains to consider the cases when $X_0$ is of types $B_n$ or $C_n$. When $X_0$ is of type $B_n$, it must be $(B_n,\{\alpha_n\})$. Since $(B_n,\{\alpha_n\})\cong (D_{n+1}, \{\alpha_n\})$ and one can always find a deformation from a group element $g\in D_{n+1}$ that does not lie in $B_n$, then the conclusion follows, whether $B_n$ is extended to $B_m$ with $m>n$, or $F_4$ (in $n=3$ case). When $X_0$ is of type $C_n$, it must be $(C_n,\{\alpha_k\}) (1\leq k\leq n-1)$. From Lemma \ref{C_n_exceptional}, $X_1\subsetneq Y\subset X_0$ for a maximal linear subspace $Y\subset X_0$. Since $\mathrm{Aut}_0(Y)=A_{2n-2k+1}$ and one can always find a deformation from a group element $g\in A_{2n-2k+1}$ that does not lie in $C_n$, then the conclusion follows, whether $C_n$ is extended to $C_m$ with $m>n$, or $F_4$ (in $n=3$ case).
\end{proof}

\begin{lemma}\label{induced_base_deform}  
Let $J\subset K \subset R$ and $\pi: G/P_J \rightarrow G/P_K$ be the canonical projection map. Let $X^{J}(w) \subset G/P_J$ be a smooth Schubert variety. Then the image $\pi(X^J(w))$ is a Schubert veriety in $G/P_K$, written as $X^K(v)$. Suppose the restricted map $\phi=\pi|_{X^J(w)}: X^J(w)\rightarrow X^K(v)$ is smooth. Let $\sM_t\subset G/P_J$, $t\in\Delta$, be a holomorphic family such that $\sM_0=X^J(w)$. Shrinking $\Delta$ around the origin if necessary, we get a holomorphic family $\sN_t:=\pi(\sM_t)\subset G/P_K$, $t\in\Delta$, such that $\sM_t\rightarrow\sN_t$ is a holomorphic map for each $t$.
\end{lemma}

\begin{proof}
We can regard $\sM$ as a closed submanifold of $G/P_J\times\Delta$ in a natural way such that under the identification of $G/P_J\times\{t\}$ with $G/P_J$ we have $\sM\cap(G/P_J\times\{t\})=\sM_t$. In particular, $\sM_0=X^J(w)\times\{0\}$. Let $\Phi$ be the composition map $\sM\subset G/P_J\times\Delta\xrightarrow{\pi\times{\rm id}} G/P_K\times\Delta$. Since $\sM$ is projective over $\Delta$, its image via $\Phi$ is also projective over $\Delta$, denoted by $\sN\subset G/P_K\times\Delta$. 

Let $Z_x$ be the fiber of $\phi$ at an arbitrary point $x\in X^K(v)$. Then the normal bundles $N_{Z_x/X^J(w)}\cong\sO_{Z_x}^m$ and $N_{Z_x/\sM}\cong\sO_{Z_x}^{m+1}$, where $m:=\dim X^K(v)$ and $\sM_0=X^J(w)\times\{0\}$ is identified with $X^J(w)$. Hence the deformations of $Z_x$ cover $\sM$. Since $Z_x$ is smooth by assumption, there exists an open neighborhood $W_x$ of $Z_x\subset\sM$ such that through any point $y\in W_x$ there exists a smooth deformation $Z(y)$ of $Z_x$ such that $y\in Z(y)\subset\sM$. As $Z_x$ is sent by $\Phi$ to a single point, so is $Z(y)$. In other words, $Z(y)$ is contained in the fiber $\Phi^{-1}\Phi(y)$. By semicontinuity of the dimension function on fibers, we have $\dim Z(y)=\dim Z_x=\dim\Phi^{-1}\Phi(y)$, and $Z(y)$ is indeed an irreucible component of the fiber $\Phi^{-1}\Phi(y)$. Conversely, each irreducible component of $\Phi^{-1}\Phi(y)$ is a smooth deformation $Z(y')$ of $Z_x\subset\sM$ for a general point $y'$ in this component. 

Let $F$ be the family of smooth deformations of $Z_x$ lying in $\sM$. Denote by $\varrho: U\rightarrow F$ be the universal map, and by $\mu: U\rightarrow\sM$ be the evaluation map. Then $\mu': U'\rightarrow W_x$ is a birational map, where $U':=\mu^{-1}(W_x)\subset U$ and $\mu'$ is the restriction of $\mu$ to $U'$. Let $E\subset W_x$ be the branch locus of $\mu'$. If $E$ is nonempty, then it has codimension at least one in $W_x$. Since $\phi: X^J(w)\rightarrow X^K(v)$ is a smooth map, we have $E\cap\sM_0=\emptyset$. Then there is an open neighborhood $\Delta'$ of $0\in\Delta$ such that for each $t\in\Delta$ we have $\mu^{-1}(\sM_t)\subset U'\setminus \mu^{-1}(E)$. Shrink $\Delta$ to $\Delta'$, we may assume that $U=U'\cong\sM$ and thus $\varrho: U\rightarrow F$ coincides with $\sM\rightarrow\sN$. The conclusion follows.
\end{proof}

\begin{lemma}\label{pic_num_one}
    Let $X^J(w) \subset X=G/P_J$ be a smooth Schubert variety of Picard number one. Let $X_1=P_K/P_J \subset G/P_J$ be a rational homogeneous space of Picard number one containing $X^J(w)$ as a Schubert variety. Then $X^{J}(w)$ is homologically rigid in $X$ if  $(X^J(w), X_1)$ is not an exceptional pair in the sense of Definition \ref{exceptional_pair}.
\end{lemma}
\begin{proof}
   Suppose $Z\subset G/P_J$ is homologically equivalent to $X^J(w)$. Then they are numerically equivalent. Let $\pi: G/P_J\rightarrow G/P_K$ be the natural projection map. Take any point $x\in\pi(Z)\subset G/P_K$. Let $\mathcal{D}$ be the set of very ample divisors on $G/P_K$ containing the point $x$. Take $D\in\mathcal{D}$ general. Then $\phi^*\mathcal{O}(D)=\mathcal{O}(\pi^{-1}(D))$ is a base point free divisor on $G/P_J$ satisfying that $Z\cap\pi^{-1}(D)\neq\emptyset$. On the other hand, the numerical class of intersection cycles $\pi^*D\cdot Z = \pi^*D\cdot X^J(w)=0$, where the last equality follows from the fact $X^J(w)$ is contained in a single fiber. Combining with the fact $Z\cap\pi^{-1}(D)\neq\emptyset$, we know that $Z\subset\pi^{-1}(D)$. In summary, we have $Z\subset\cap_{D\in\mathcal{D}_g} \pi^{-1}(D)=\pi^{-1}(x)$, where $\mathcal{D}_g$ is the set of general elements in $\mathcal{D}$. Take $g\in G$ such that $g\cdot x$ is the base point $o\in G/P_K$. Then $g\cdot Z\subset X_1$ and the homological rigidity of $(X^{J}(w), X)$ follows from that $(X^J(w), X_1)$ is not an exceptional pair.
\end{proof}

\begin{lemma}\label{pic_one_base}
     Let $X^J(w) \subset X=G/P_J$ be a smooth Schubert variety of Picard number one. Let $\alpha\notin J$ be a simple root and  $\pi:X\rightarrow Y=G/P_{K} $ be the natural projection map for some $K\subset R\backslash \{\alpha\}$. Suppose that $X^J(w)$ is mapped isomorphically into a Schubert variety $Y_0\subset Y$ through $\pi$ and 
     it satisfies 
     \begin{enumerate}
         \item $(Y_0,Y)$ is not an exceptional pair in the sense of Definition \ref{exceptional_pair} if $Y$ is of Picard number one (i.e., $K=R\backslash \{\alpha\}$);
         \item $Y_0$ is homologically rigid in $Y$ if $Y$ is of Picard number at least two,
     \end{enumerate} then $X^{J}(w)$ is homologically rigid in $X$.
\end{lemma}
\begin{proof}
    Let $\pi_\alpha: X=G/P_J \rightarrow G/P_{J\cup \{\alpha\}}$ be the natural projection map. Since $X^J(w)$ is mapped isomorphically into the Schubert variety $Y_0\subset Y$ through $\pi$, without loss of generality we may assume that $X^J(w)$ lies in the central fiber $\pi_\alpha^{-1}(o)=P_{J\cup \{\alpha\}}/P_J$, where $o\in G/P_{J\cup \{\alpha\}}$ denotes the base point. Moreover, $\pi_\alpha^{-1}(o)$ is mapped isomorphically into the homogeneous subspace associated with the subdiagram $\Gamma_{J\cup \{\alpha\}}(\alpha)$ in $Y=G/P_{K}$ through $\pi$. According to our condition on $(Y_0,Y)$, $(Y_0, \pi(\pi_\alpha^{-1}(o)))$ is not an exceptional pair, from Lemma \ref{exceptional_triple}. Again, since $\pi|_{\pi_\alpha^{-1}(o)}$ is an isomorphism, $X^{J}(w)$ is contained in the rational homogeneous  subspace $P_{J\cup \{\alpha\}}/P_J\subset G/P_J$ of Picard number one, then the rest can be done by Lemma \ref{pic_num_one}.
       \begin{eqnarray*}
\xymatrix{& G/P_{J}\ar[ld]_{\pi}\ar[rd]^{\pi_\alpha}  & \\
Y=G/P_{K}  & &G/P_{J\cup \{\alpha\}}  \\}
\end{eqnarray*}
\end{proof}
\subsection{The cases of higher Picard numbers}
The following result gives a criterion on the homological rigidity of Schubert varieties via the known cases of Picard number one.

\begin{proposition} \label{prop homology rigidity reduction}
Let $X^J(w) \subset X=G/P_J$ be a smooth Schubert variety in $X=G/P_J$. If there exists a tower of Billey-Postnikov decompositions corresponding to the iterated fiber bundles $X^{J}(w)=X_0\rightarrow X_1 \rightarrow X_2 \rightarrow \cdots \rightarrow X_m=pt$, such that  in each step the central fiber is a smooth Schubert variety in a rational homogeneous space of Picard number one and all pairs in the fibers of the iterated fiber bundle are not exceptional in the sense of Definition \ref{exceptional_pair} (all algebraic groups are chosen to be subgroups of $G$), then $X^J(w)$ has homological rigidity in $X$.
\end{proposition}

\begin{proof}
Let $X=G/P_{J}$, $X^J(w)\subset X$ be a smooth Schubert variety. We prove by induction on \textbf{the Picard number $\rho$ of $X^{J}(w)$}. If $\rho$ is one, the conclusion follows from Lemma \ref{pic_num_one} and Lemma \ref{pic_one_base}. In fact, from Theorem \ref{smooth_BP}, there exists a root $\alpha\in I=J^c$ such that the projection $G/P_J\rightarrow G/P_{J\cup \{\alpha\}}$ sends $X_0$ to $X_1$, where $X_i(0\leq i \leq m)$ is given in the statement of the Proposition. If $X^{J}(w)=X_0$ is contracted in the first projection $G/P_J\rightarrow G/P_{J\cup \{\alpha\}}$, we obtain the conclusion from Lemma \ref{pic_num_one}. Otherwise, the first projection is an isomorphism on $X^{J}(w)$, then we obtain the conclusion from Lemma \ref{pic_one_base}, together with an induction argument on the Picard number of $G/P_J$.

If $\rho(X^{J}(w))=k\geq 2$, we assume that the conclusion holds for the cases when $\rho\leq k-1$. 
From Theorem \ref{smooth_BP} there exists a root $\alpha\in I=J^c$ and a Billey-Postnikov decomposition $w=vu$ with respect to $(J, K=J\cup\{\alpha\})$ such that
\begin{itemize}
\item[(i)] the projection $\pi: G/P_{J} \rightarrow G/P_{K}$ sends $X^J(w)$ onto $X^K(v)$,
\item[(ii)] $X^K(v)$ is a smooth Schubert variety of Picard number one in $G/P_K$,
\item[(iii)] $X^{J}(u)$ is a smooth Schubert variety in $P_K/P_J \cong G_K/G_K\cap P_J$ with Picard number $k-1$.
\end{itemize}
In summary, we have the commutative diagram as follows:
\begin{eqnarray*}
\xymatrix{X^J(u)\ar@{^(->}[r]\ar@{^(->}[d] & X^J(w)\ar@{^(->}[d]\ar@{->>}[r] & X^K(v)\ar@{^(->}[d] \\
P_K/P_J\ar@{^(->}[r] & G/P_J\ar@{->>}^-{\pi}[r] & G/P_K.
}
\end{eqnarray*}

Take a holomorphic family $\sM_t\subset G/P_J$, $t\in\Delta$, such that $\sM_0=X^J(w)$. By Lemma \ref{induced_base_deform}, we can shrink $\Delta$ around the origin to obtain a holomorphic family $\sN_t:=\pi(\sM_t)\subset G/P_K$, $t\in\Delta$, such that $\sN_0=X^K(v)$. From the base case $\rho=1$ of the induction, there exists a holomorphic map $h: \Delta\rightarrow G$ with $h(0)=\operatorname{id}$ such that $h(t)\cdot\sN_t= X^K(v)$ as closed subvarieties of $G/P_K$ for all $t\in\Delta$.

As in the proof of Lemma \ref{induced_base_deform}, we write $\sM$ and $\sN$ as closed subvarieties of $G/P_J\times\Delta$ and $G/P_K\times\Delta$ respectively, and the the projection $\pi: G/P_J\rightarrow G/P_K$ induces the smooth map $\Phi: \sM\rightarrow\sN$.
Set $\sM':=h\cdot \sM$ and $\sN':=h\cdot\sN$, i.e., $\sM'_t=h(t)\cdot\sM_t$ and $\sN'_t=h(t)\cdot\sN_t$. Then $\sN'=X^K(v)\times\Delta$ as a closed subvariety of $G/P_K\times\Delta$, and we have a commutative diagram:
\begin{eqnarray*}
\xymatrix{\sM'\ar@{^(->}[r]\ar@{->>}_-{\Phi'}[d] & \sS_J\ar@{->>}^-{\pi\times{\rm id}}[d] \\
\sN'\ar@{^(->}[r] & \sS_K.
}
\end{eqnarray*}
where $\sS_J:=G/P_J\times\Delta$ and $\sS_K:=G/P_K\times\Delta$. Now the holomorphic map $\Phi'$ induces a holomorphic map
\begin{eqnarray*}
\psi: & \sN'\longrightarrow {\rm Chow}(\sS_J/\Delta) \\
& \xi\mapsto(\Phi')^{-1}(\xi).
\end{eqnarray*}
The projection $\sS_J\rightarrow G/P_J$ induces a holomorphic map 
$$\varphi: {\rm Chow}(\sS_J/\Delta)\rightarrow {\rm Chow}(G/P_J).$$
Since $\sN'=X^K(v)\times\Delta$, we get a holomorphic family of holomorphic maps 
\begin{eqnarray*}
\gamma_t: & X^K(v)\rightarrow {\rm Chow}(G/P_J) \\
& z\mapsto\varphi(\psi(z, t)),
\end{eqnarray*}
where $t\in\Delta$. Denote by $F_t$ the image of $\gamma_t$, by $\varrho_t: U_t\rightarrow F_t$ the universal map, and by $\mu_t: U_t\rightarrow G/P_J$ the evaluation map. We can conclude from the commutative diagram
\begin{eqnarray*}
\xymatrix{ U_t\ar^-{\mu_t\times{\rm id}}[r]\ar_-{\varrho_t}[d] & \sM'_t\ar^-{\Phi'_t}[d] \\
F_t & \sN'_t\ar^-{\gamma_t\times{\rm id}}[l]
}
\end{eqnarray*}
that $\gamma_t$ is a closed embedding sending $X^K(v)$ onto $F_t\subset{\rm Chow}(G/P_J)$, and $\sM'_t$ is the image of $\mu_t$ under the identification $G/P_J\times\{t\}\cong G/P_J$.

We claim that, shrinking $\Delta$ around the origin if necessary, there exists a holomorhic map $g:\Delta\rightarrow G$ such that $F_t=g(t)\cdot F_0$. Recall that $\gamma_0$ sends the base point $o\in X^K(v)$ to the element $[X^J(u)]\in {\rm Chow}(G/P_J)$. Let $D$ be the connected component of $ {\rm Chow}(G/P_J)$ containing $[X^J(u)]$. By the inductive assumption, and since  exceptional pair does not appear, there exists an open neighborhood $D^o$ of $[X^J(u)]\in D$ such that $D^o\subset G\cdot[X^J(u)]$. Then we have $D=G\cdot[X^J(u)]$ by Proposition \ref{reduced_to_local_deform}. Let $Q(u)$ be the subgroup of $G$ stabilizing $X^J(u)$. Then $Q(u)$ is a parabolic subgroup of $G$, and we have $D=G/Q(u)$. 

Since $\varphi(\psi(\sN'))$ is connected, we have $F_t\subset \varphi(\psi(\sN'))\subset D$ for all $t$. Let $Q(v)$ be the subgroup of $G$ stabilizing $X^K(v)\subset G/P_K$. Then $Q(v)$ is a parabolic subgroup of $G$.
Note that $\gamma_0$ is a $Q(v)$-equivariant closed embedding, which makes $F_0\cong X^K(v)$ a Schubert subvariety of $D=G/Q(u)$. Moreover, $Q(u)\subset P_K$, and the natural projection $G/Q(u)\rightarrow G/P_K$ maps the Schubert variety $F_0$ isomorphically into the Schubert variety $X^{K}(v)\subset G/P_K$.
 Since $X^K(v)$ is of Picard number one and there is no exceptional pair, the claim follows from the base case $\rho=1$, together with Lemma \ref{pic_one_base}.

By the claim above, we have $U_t=g(t)\cdot U_0$ and $\sM'_t=g(t)\cdot \sM'_0$. By the construction of $\sM'$, we have 
\begin{eqnarray*}
&& \sM_t=h(t)^{-1}\cdot\sM'_t= (h(t)^{-1}g(t))\cdot\sM'_0  \\
&& \hspace{1cm} =(h(t)^{-1}g(t)h(0))\cdot\sM_0=(h(t)^{-1}g(t))\cdot X^J(w).
\end{eqnarray*}
This completes the proof of inductive step, by Proposition \ref{reduced_to_local_deform}. Hence $X^J(w)$ is homologically rigid by induction on its Picard number.
\end{proof}
 Now we are ready to prove Theorem \ref{main_thm_simple_ver}.

\begin{proof}[Proof of Theorem \ref{main_thm_simple_ver}]
   Existence of a sequence of Billey-Postnikov decompositions can be guaranteed by Theorem \ref{smooth_BP}. Since all marked roots in $I$ are long roots, from \cite[Proposition 3.7]{HM13}, any pairs in the fibers are associated with subdiagrams. Then by Proposition \ref{prop-exceptional-pair}, they are not exceptional and the conclusion follows from Proposition \ref{prop homology rigidity reduction}.
\end{proof}

\begin{figure}[htbp]
\centering
\includegraphics[width=0.6\textwidth]{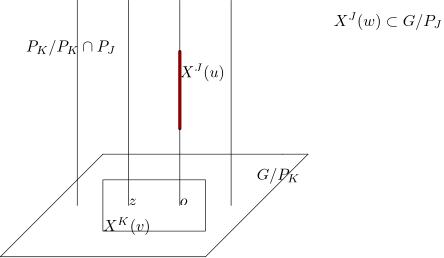}
\caption{ Billey-Postnikov decomposition}
\label{fig:example}
\end{figure}
Also, we immediately have

\begin{corollary}\label{subdiagram_rigid_marked_root_outside}
 Let $X_0 \subset X$ is a Schubert variety of subdiagram type and $L$ be the set of simple roots in the subdiagram. If $I\cap L$ contains no short roots, then $X_0$ has homological rigidity in $X$.
\end{corollary}
\begin{proof}
 Since $I\cap L$ consists of only long roots, no pairs in the fibers of Billey-Postnikov decompositions in each step are exceptional, by Proposition \ref{prop-exceptional-pair}. Then the conclusion follows  from Proposition \ref{prop homology rigidity reduction}.
\end{proof}

\begin{remark}\label{difference_high_picard}
There are some interesting differences between Picard number one and higher Picard number cases.

    From the proof of Proposition \ref{prop homology rigidity reduction}, we can see that what we need is the \textit{existence} of one tower of Billey-Postnikov decompositions such that exceptional pairs do not appear in each step. We have the freedom to choose required  Billey-Postnikov decompositions.
\begin{example}   
 In the type $C$ case, let $X=G/P_J$ with $G=C_n$ and  $I=\{\alpha_{k-1}, \alpha_k\}, J=I^c$. We assume $L=\{\alpha_j\mid k-1\leq j\leq b\}$ for some $k+1\leq b\leq  n-1$. The corresponding marked Dynkin diagram $\Gamma_R(I)$ with subdiagram $\Gamma_{L}(I\cap L)$  is as follows. 
 
      \begin{center}
    
  \begin{tikzpicture}[scale=0.5]
  \draw[thick] (3.6,0.75) -- (8.5,0.75)  ;
  \draw[thick] (3.6,-0.75) -- (8.5,-0.75)  ;
   \draw[thick] (3.6, 0.75) -- (3.6,-0.75)  ;
    \draw[thick] (8.5,0.75) -- (8.5,-0.75)  ;
	\draw[thick] (0,0) -- (1.5,0)  ;
 \draw[thick, dashed] (1.5,0) -- (3,0)  ;
 \draw[thick] (3,0) -- (4.5,0);
 \draw[thick] (6,0) -- (7.5,0)  ;
  \draw[thick] (4.5,0) -- (6,0);
  \draw[thick, dashed] (7.5,0) -- (9,0)  ;
   \draw[-<-,thick, double] (9,0) -- (10.5,0)  ;
\draw[thick] (3,0) -- (4.5,0)  ;
		\draw[ thick, fill=white] (0,0) circle (3pt) node[above, outer sep=3pt]{$\alpha_1$};
	\draw[ thick, fill=white] (1.5,0) circle (3pt)node[above, outer sep=3pt]{$\alpha_2$};
	\draw[ thick, fill=white] (3,0) circle (3pt) node{};
	\draw[ thick, fill=black] (4.5,0) circle (3pt) node[above,outer sep=3pt]{$\alpha_{k-1}$};
 \draw[ thick, fill=black] (6,0) circle (3pt) node[above,outer sep=3pt]{$\alpha_k$};
\draw[ thick, fill=white] (7.5,0) circle (3pt) node[above, outer sep=3pt]{};
\draw[ thick, fill=white] (9,0) circle (3pt) node[above, outer sep=3pt]{};
 \draw[ thick, fill=white] (10.5,0) circle (3pt) node[above,outer sep=3pt]{$\alpha_n$};
	\end{tikzpicture} 	
   \end{center}
  Then if we take $K=J\cup \{\alpha_{k}\}$ and consider the Billey-Postnikov decomposition with respect to $(J, K)$, exceptional pair would appear (see Proposition \ref{prop-exceptional-pair}); but if we take $K'=J\cup \{\alpha_{k-1}\}$ and consider the   Billey-Postnikov decomposition with respect to $(J, K')$, then there is no exceptional pair and the homological rigidity holds from Proposition \ref{prop homology rigidity reduction}.
   
\end{example}   
In the type $B$ case, in contrast to the case of Picard number one, there exist non-rigid smooth Schubert varieties in higher Picard numbers.

\begin{example}\label{TypeB_nonrigid_ex}
    Let $X=G/P_J$ with $G=B_n$ and $I=\{\alpha_k, \alpha_n\}, J=I^c$. We assume that $L$ is the union of $\{\alpha_k\}$ with $\{\alpha_j\mid k+3\leq j\leq n\}$. The corresponding marked Dynkin diagram $\Gamma_R(I)$ with subdiagram $\Gamma_{L}(I)$ is as follows. 
    \begin{center}
  \begin{tikzpicture}[scale=0.5]
  \draw[thick] (2.1,0.75) -- (3.9, 0.75)  ;
  \draw[thick] (2.1,-0.75) -- (3.9,-0.75)  ;
   \draw[thick] (2.1, 0.75) -- (2.1,-0.75)  ;
    \draw[thick] (3.9,0.75) -- (3.9,-0.75)  ;
\draw[thick] (6.6,0.75) -- (11.4, 0.75)  ;
  \draw[thick] (6.6,-0.75) -- (11.4,-0.75)  ;
   \draw[thick] (6.6, 0.75) -- (6.6,-0.75)  ;
    \draw[thick] (11.4,0.75) -- (11.4,-0.75)  ;
    \draw[thick, dashed] (5.1,0.85) -- (12,0.85)  ;
  \draw[thick, dashed] (5.1,-0.85) -- (12,-0.85)  ;
   \draw[thick, dashed] (5.1, 0.85) -- (5.1,-0.85)  ;
    \draw[thick, dashed] (12,0.85) -- (12,-0.85)  ;

    \draw[thick] (-1.5,0) -- (0,0)  ;
	\draw[thick, dashed] (0,0) -- (1.5,0)  ;
 \draw[thick] (1.5,0) -- (3,0)  ;
 \draw[thick] (3,0) -- (4.5,0);
 \draw[thick] (6,0) -- (7.5,0)  ;
  \draw[thick] (4.5,0) -- (6,0);
  \draw[thick, dashed] (7.5,0) -- (9,0)  ;
   \draw[->-,thick, double] (9,0) -- (10.5,0)  ;
\draw[thick] (3,0) -- (4.5,0)  ;
\draw[ thick, fill=white] (-1.5,0) circle (3pt) node[above, outer sep=3pt]{$\alpha_1$};
		\draw[ thick, fill=white] (0,0) circle (3pt) node[above, outer sep=3pt]{$\alpha_2$};
	\draw[ thick, fill=white] (1.5,0) circle (3pt)node{};
	\draw[ thick, fill=black] (3,0) circle (3pt) node[above,outer sep=3pt]{$\alpha_{k}$};
	\draw[ thick, fill=white] (4.5,0) circle (3pt) node{};
 \draw[ thick, fill=white] (6,0) circle (3pt) node{};
\draw[ thick, fill=white] (7.5,0) circle (3pt) node[above, outer sep=3pt]{};
\draw[ thick, fill=white] (9,0) circle (3pt) node[above, outer sep=3pt]{};
 \draw[ thick, fill=black] (10.5,0) circle (3pt) node[above,outer sep=3pt]{$\alpha_n$};
	\end{tikzpicture} 	
   \end{center}  
  
The Schubert variety $X^J(w)$ in this case comes from a product of subdiagrams $(A_1\times B_{n-k-2}, \{\alpha_k\}\times \{\alpha_n\})$. The second factor $(B_{n-k-2}, \{\alpha_n\})$ can be included in $(B_{n-k-1}, \{\alpha_n\})$ (the subdiagram in the dashed box). Since the automorphism group of $(B_{n-k-1}, \{\alpha_n\})$ is $D_{n-k}$, we can take a local deformation of the second factor by a group element  $g\in D_{n-k}$ that does not lie in $B_{n-k-1}$ (see Remark \ref{covering_pair}). This will give a local deformation of $X^J(w)$ in $X$ not from the action of $B_n$, and hence the Schubert variety is not homologically rigid.
\end{example}

\end{remark}

\subsection{The cases of subdiagram type}

We then give a proof of Theorem \ref{homological_subdiagram}, which gives a complete classification of homologically rigid/non-rigid Schubert varieties of subdiagram type. From Proposition \ref{prop-exceptional-pair}, we know exceptional pairs appear in the cases when $G=B_n, C_n, F_4$ or $G_2$. Since $G_2$ case is clear from Proposition \ref{prop homology rigidity reduction}, in this subsection, we discuss the proof of Theorem \ref{homological_subdiagram} based on a case-by-case discussion for $B_n,C_n,F_4$. For types $B_n$ and $C_n$, we first deal with those exceptions and then turn to prove the rigid cases. For type $F_4$, we deal with the cases where all marked roots in $I$ are contained in the subdiagram and then turn to deal with the rest. 
 
 In this subsection, we assume that the Schubert variety $X_0=X^J(w)\subset X=G/P_J$ is associated with a subdiagram $\Gamma_L(I\cap L)$ of the marked Dynkin diagram $\Gamma(I)$, where $I=J^c$.

\subsubsection{$B_n$-cases}\label{Bn_subdiagram}
The following lemma is for Case (8) in Theorem \ref{homological_subdiagram}.
\begin{lemma}\label{type_B_subdiagram_not_rigid}
 Let $X=(B_n, \{\alpha_n\} \cup R_1)(n\geq 3), L=\{\alpha_{b+1},\alpha_{b+2}, \cdots, \alpha_n\}\cup R_2$,
         $1<  b <n$, $R_1 \subset \{\alpha_1, \alpha_2, \dots, \alpha_{b-1}\}, R_2\subset \{\alpha_1, \alpha_2, \dots, \alpha_{b-2}\}$($R_2=\emptyset$ if $b=2$), $R_1\neq \emptyset$. Then $X_0=X^J(w)$ is not homologically rigid.
\end{lemma}
\begin{proof}
We have the marked Dynkin diagram and subdiagram as follows.
 \begin{center}
  \begin{tikzpicture}[scale=0.5]
  \draw[thick] (2.1,0.75) -- (3.9, 0.75)  ;
  \draw[thick] (2.1,-0.75) -- (3.9,-0.75)  ;
    \draw[thick] (3.9,0.75) -- (3.9,-0.75)  ;
\draw[thick] (6.6,0.75) -- (11.4, 0.75)  ;
  \draw[thick] (6.6,-0.75) -- (11.4,-0.75)  ;
   \draw[thick] (6.6, 0.75) -- (6.6,-0.75)  ;
    \draw[thick] (11.4,0.75) -- (11.4,-0.75)  ;
  \draw[thick, dashed] (5.4,0.85) -- (11.6, 0.85)  ;
  \draw[thick, dashed] (5.4,-0.85) -- (11.6,-0.85)  ;
   \draw[thick, dashed] (5.4, 0.85) -- (5.4,-0.85)  ;
    \draw[thick, dashed] (11.6,0.85) -- (11.6,-0.85)  ;

    \draw[thick] (-1.5,0) -- (0,0)  ;
	\draw[thick, dashed] (0,0) -- (1.5,0)  ;
 \draw[thick] (1.5,0) -- (3,0)  ;
 \draw[thick] (3,0) -- (4.5,0);
 \draw[thick] (6,0) -- (7.5,0)  ;
  \draw[thick, dashed] (4.5,0) -- (6,0);
  \draw[thick, dashed] (7.5,0) -- (9,0)  ;
   \draw[->-,thick, double] (9,0) -- (10.5,0)  ;
\draw[thick] (3,0) -- (4.5,0)  ;
\draw[ thick, fill=white] (-1.5,0) circle (3pt) node[above, outer sep=3pt]{$\alpha_1$};
		\draw[ thick, fill=white] (0,0) circle (3pt) node[above, outer sep=3pt]{$\alpha_2$};
	\draw[ thick, fill=black] (1.5,0) circle (3pt)node[above,outer sep=3pt]{};
	\draw[ thick, fill=white] (3,0) circle (3pt) node{};
	\draw[ thick, fill=white] (4.5,0) circle (3pt) node[above,outer sep=3pt]{};
 \draw[ thick, fill=white] (6,0) circle (3pt) node[above,outer sep=3pt]{$\alpha_b$};
\draw[ thick, fill=white] (7.5,0) circle (3pt) node[above, outer sep=3pt]{};
\draw[ thick, fill=white] (9,0) circle (3pt) node[above, outer sep=3pt]{};
 \draw[ thick, fill=black] (10.5,0) circle (3pt) node[above,outer sep=3pt]{$\alpha_n$};
	\end{tikzpicture} 	
   \end{center}  
    The proof is similar to the discussion of Example \ref{TypeB_nonrigid_ex}. Note that condition $R_1\neq \emptyset$ implies $B_n=\operatorname{Aut}_0(X)$, and $\alpha_{b-1}\not\in L$ implies the existence of deformation not from $B_n=\operatorname{Aut}_0(X)$.
\end{proof}
\begin{proof}[Proof of Theorem \ref{homological_subdiagram} when $G=B_n$]

When $G=B_n$, if the short simple root $\alpha_n \not\in I$ (i.e., it is not marked), or $\alpha_n\not \in L$, it is easy to see from Proposition \ref{prop-exceptional-pair} that there is no exceptional pair in any Billey-Postnikov decomposition and Proposition \ref{prop homology rigidity reduction} works well. When $|I|=1$, homological rigidity holds from Theorem \ref{HM_homological}.

In the following we assume that $\alpha_n \in I \cap L$ and $|I|\geq 2$.  Let $\alpha_s$ be the rightest marked simple root in $I\backslash \{\alpha_n\}$. Let $\pi: X=G/P_J\rightarrow G/P_K$ be the natural projection with $K=R\backslash \{\alpha_s\}$. Then the image $Y_0:=\pi(X_0)$ is again a Schubert variety of subdiagram type on $G/P_K$. By Proposition \ref{prop-exceptional-pair}, $(Y_0, G/P_K)$ is not an exceptional pair and this is a Billey-Postnikov decomposition. Let $\Gamma_{K'}$ be the connected component of $\Gamma_K$ such that the short simple root $\alpha_n\in K'$. 

If $K'$ is contained in $L$, then the $\pi$-fibers of $(X_0, X)$ are associated with Dynkin diagrams whose connected components are of type $A$ or $B$. Moreover, the $B$-factors of the $\pi$-fibers of $X_0$ and $X$ coincide, both being the marked Dynkin diagram $\Gamma_{K'}(\alpha_n)$. Then Proposition \ref{prop-exceptional-pair} implies that we can take Billey-Postnikov decomposition of the $\pi$-fibers of $(X_0, X)\rightarrow (Y_0, G/P_K)$ such that  there are no exceptional pairs. Then by Proposition \ref{prop homology rigidity reduction}, $X_0$ has homological rigidity in $X$.

Now we assume  $K'$ is NOT contained in $L$. Set $b:=\mbox{max}\{i\mid \alpha_i\not\in L\}$. 
By assumption, $s\leq b-1$. If $\alpha_{b-1}\in L$, then by Lemma \ref{type_B_subdiagram_special_rigid} below the Schubert variety $X_0$ has homological rigidity in $X$. If $\alpha_{b-1}\not\in L$, then we are in the situation of Theorem \ref{homological_subdiagram} (8), in which case the Schubert variety $X_0$ is not homologically rigid by Lemma  \ref{type_B_subdiagram_not_rigid}.   
\end{proof} 
    
\begin{lemma}\label{type_B_subdiagram_special_rigid}
 Let $X=(B_n, \{\alpha_n\} \cup R_1)(n\geq 3), L= \{\alpha_{b+1},\alpha_{b+2}, \cdots, \alpha_n\}\cup \{\alpha_{b-1}\}\cup R_2$,
         $1<  b <n$, $R_1 \subset \{\alpha_1, \alpha_2, \dots, \alpha_{b-1}\}, R_2\subset \{\alpha_1, \alpha_2, \dots, \alpha_{b-2}\}$($R_2=\emptyset$ if $b=2$), $R_1\neq \emptyset$. Then $X_0=X^J(w)$ is homologically rigid.
\end{lemma}
\begin{proof}
We have the marked Dynkin diagram and subdiagram as follows.
 \begin{center}
  \begin{tikzpicture}[scale=0.5]
  \draw[thick] (2.1,0.75) -- (3.9, 0.75)  ;
  \draw[thick] (2.1,-0.75) -- (3.9,-0.75)  ;
    \draw[thick] (3.9,0.75) -- (3.9,-0.75)  ;
\draw[thick] (5.1,0.75) -- (11.4, 0.75)  ;
  \draw[thick] (5.1,-0.75) -- (11.4,-0.75)  ;
   \draw[thick] (5.1, 0.75) -- (5.1,-0.75)  ;
    \draw[thick] (11.4,0.75) -- (11.4,-0.75)  ;
  \draw[thick, dashed] (4.3,0.85) -- (11.6, 0.85)  ;
  \draw[thick, dashed] (4.3,-0.85) -- (11.6,-0.85)  ;
   \draw[thick, dashed] (4.3, 0.85) -- (4.3,-0.85)  ;
    \draw[thick, dashed] (11.6,0.85) -- (11.6,-0.85)  ;

    \draw[thick] (-1.5,0) -- (0,0)  ;
	\draw[thick, dashed] (0,0) -- (1.5,0)  ;
 \draw[thick] (1.5,0) -- (3,0)  ;
 \draw[thick] (3,0) -- (4.5,0);
 \draw[thick] (6,0) -- (7.5,0)  ;
  \draw[thick] (4.5,0) -- (6,0);
  \draw[thick, dashed] (7.5,0) -- (9,0)  ;
   \draw[->-,thick, double] (9,0) -- (10.5,0)  ;
\draw[thick] (3,0) -- (4.5,0)  ;
\draw[ thick, fill=white] (-1.5,0) circle (3pt) node[above, outer sep=3pt]{$\alpha_1$};
		\draw[ thick, fill=white] (0,0) circle (3pt) node[above, outer sep=3pt]{$\alpha_2$};
	\draw[ thick, fill=black] (1.5,0) circle (3pt)node[above,outer sep=3pt]{};
	\draw[ thick, fill=white] (3,0) circle (3pt) node{};
	\draw[ thick, fill=white] (4.5,0) circle (3pt) node[above,outer sep=3pt]{$\alpha_b$};
 \draw[ thick, fill=white] (6,0) circle (3pt) node{};
\draw[ thick, fill=white] (7.5,0) circle (3pt) node[above, outer sep=3pt]{};
\draw[ thick, fill=white] (9,0) circle (3pt) node[above, outer sep=3pt]{};
 \draw[ thick, fill=black] (10.5,0) circle (3pt) node[above,outer sep=3pt]{$\alpha_n$};
	\end{tikzpicture} 	
   \end{center}  
One can see that this is the borderline case (see Lemma \ref{type_B_subdiagram_not_rigid}). We can always find exceptional pairs in any Billey-Postnikov decompositions, so Proposition \ref{prop homology rigidity reduction} does not apply; On the other hand, we can not construct a local deformation not from the group action of $B_n$ as in Example \ref{TypeB_nonrigid_ex} or Lemma \ref{type_B_subdiagram_not_rigid} since $\alpha_{b-1}\in L$.

Now we prove the homological rigidity for this case making use of the product structure of the Schubert varieties, together with both the statement and proof of Proposition \ref{prop homology rigidity reduction}. Let $K=J\cup \{\alpha_n\}$ and take the Billey-Postnikov decomposition $w=vu$ with respect to $(J,K)$. Denote by $\pi: G/P_J\rightarrow G/P_K$ the natural projection. Note that in this case $X_0\cong X^{K}(v) \times X^{J}(u)$, we still use $X^{K}(v)$ to denote the Schubert variety in $X=G/P_J$ associated with the collection of connected components in $L$ without $\alpha_n$. 

In the base level, any further Billey-Postnikov decomposition of $X^{K}(v)\subset G/P_K$ has no exceptional pair, from Proposition \ref{prop-exceptional-pair} (since all marked roots in $R_1=I\setminus \{\alpha_n\}$ are long roots). Hence $X^{K}(v)\subset G/P_K$ is homologically rigid, by Proposition \ref{prop homology rigidity reduction} or Theorem \ref{main_thm_simple_ver}. From the proof of Proposition \ref{prop homology rigidity reduction}, without loss of generality we may assume that the local deformation $\mathcal{M} \subset G/P_J\times \Delta$ induces a trivial deformation of $X^{K}(v)$ in the base level. 
 Shrinking $\Delta$ around the origin, then for each $t\in \Delta$, $\mathcal{M}_t$ is still a product. We may write $\mathcal{M}_t=X^K(v)\times Y_t$ with $\mathcal{M}_0=X^{K}(v)\times X^J(u)$ and $Y_t\subset P_K/P_K\cap P_J$. Define \[ \gamma_t:Y_t\rightarrow \operatorname{Chow}(G/P_J)\]
by $\gamma_t(y)=[X^{K}(v)\times \{y\}]$ for $y\in Y_t$. From Lemma \ref{stab.subdiagram} and Corollary \ref{subdiagram_rigid_marked_root_outside}, by identifying $X^{K}(v)$ as a Schubert variety of subdiagram type in $X$ (which is homologically rigid), the Chow space of $X^{K}(v)$ in $X=G/P_J$ is $G/P_{R\backslash R_3}$ where $R_3=(R_1\backslash L) \cup R'_2\cup\{\alpha_b,\alpha_n\}$, with $R'_2\subset \{\alpha_1,\alpha_2,\cdots, \alpha_{b-2}\}$ being the subset of simple roots in $\{\alpha_1,\alpha_2,\cdots, \alpha_{b-2}\}$ adjacent to $\Gamma_L$. 

Now, $Y_t$ is a local deformation of $Y_0=X^{J}(u)$ in $G/P_{R\backslash R_3}$. 
Consider the natural projection $\phi: G/P_{R\backslash R_3}\rightarrow G/P_{(R\backslash R_3)\cup\{\alpha_n\}}$. Then $X^J(u)$ coincides with the fiber of $\phi$ at the base point. Hence $X^{J}(u)$ has homological rigidity in $G/P_{R\backslash R_3}$ and there is a holomorphic map $g:\Delta \rightarrow G$ such that $Y_t=g(t)\cdot Y_0$. Hence $\mathcal{M}_t=g(t)\cdot X_0$. Note that we have assumed that the deformation at the base level is trivial, $g(t)$ stabilizes $X^{K}(v)\subset G/P_K$ and hence $g(t)\in 
P_{R\backslash R_4}$ where $R_4=R_3\backslash \{\alpha_n\}$. Thus $g(t)$ actually stabilizes $X_0 \subset X=G/P_J$, i.e., $\mathcal{M}$ is also a trivial family assuming the induced family at the base level is trivial.
\end{proof}

\subsubsection{$C_n$-cases}\label{Cn_subdiagram}
The following lemma is for Case (7) in Theorem \ref{homological_subdiagram}.
\begin{lemma}\label{Type_C_subdiagram_not_rigid}
    Let $X=(C_n,\{\alpha_k\}\cup R_1)(n\geq 3), L=\{\alpha_{k},\alpha_{k+1}\cdots, \alpha_{b-1}\}\cup R_2$,
         $2\leq k< b\leq n$, $R_1\subset \{\alpha_1,\alpha_2,\dots, \alpha_{k-2}, \alpha_{k-1}\}, R_2\subset \{\alpha_1,\alpha_2,\dots, \alpha_{k-2}\}$ ($R_2=\emptyset$ if $k=2$). Then $X_0=X^J(w)$ is not homologically rigid.
\end{lemma}
\begin{proof}
We have the marked Dynkin digram and subdigram as follows.
 \begin{center}
  \begin{tikzpicture}[scale=0.5]
  \draw[thick] (2.1,0.75) -- (3.9, 0.75)  ;
  \draw[thick] (2.1,-0.75) -- (3.9,-0.75)  ;
    \draw[thick] (3.9,0.75) -- (3.9,-0.75)  ;
\draw[thick] (5.1,0.75) -- (8.4, 0.75)  ;
  \draw[thick] (5.1,-0.75) -- (8.4,-0.75)  ;
   \draw[thick] (5.1, 0.75) -- (5.1,-0.75)  ;
    \draw[thick] (8.4,0.75) -- (8.4,-0.75)  ;
    \draw[thick, dashed] (4.8,0.85) -- (12.5,0.85)  ;
  \draw[thick, dashed] (4.8,-0.85) -- (12.5,-0.85)  ;
   \draw[thick, dashed] (4.8, 0.85) -- (4.8,-0.85)  ;
    \draw[thick, dashed] (12.5,0.85) -- (12.5,-0.85)  ;

    \draw[thick] (-1.5,0) -- (0,0)  ;
	\draw[thick, dashed] (0,0) -- (1.5,0)  ;
 \draw[thick] (1.5,0) -- (3,0)  ;
 \draw[thick] (3,0) -- (4.5,0);
 \draw[thick,dashed] (6,0) -- (7.5,0)  ;
  \draw[thick] (4.5,0) -- (6,0);
  \draw[thick] (7.5,0) -- (9,0)  ;
   \draw[thick,dashed] (9,0) -- (10.5,0)  ;
   \draw[-<-,thick, double] (10.5,0) -- (12,0)  ;
\draw[thick] (3,0) -- (4.5,0)  ;
\draw[ thick, fill=white] (-1.5,0) circle (3pt) node[above, outer sep=3pt]{$\alpha_1$};
		\draw[ thick, fill=white] (0,0) circle (3pt) node[above, outer sep=3pt]{$\alpha_2$};
	\draw[ thick, fill=black] (1.5,0) circle (3pt)node[above, outer sep=3pt]{};
	\draw[ thick, fill=white] (3,0) circle (3pt)node{};
	\draw[ thick, fill=white] (4.5,0) circle (3pt) node{};
 \draw[ thick, fill=black] (6,0) circle (3pt) node[above, outer sep=3pt]{$\alpha_k$};
\draw[ thick, fill=white] (7.5,0) circle (3pt) node[above, outer sep=3pt]{};
\draw[ thick, fill=white] (9,0) circle (3pt) node[above, outer sep=3pt]{};
 \draw[ thick, fill=white] (10.5,0) circle (3pt) node[above,outer sep=3pt]{};
 \draw[ thick, fill=white] (12,0) circle (3pt) node[above,outer sep=3pt]{$\alpha_n$};
	\end{tikzpicture} 	
   \end{center}  

   We can see that in this case the Schubert variety is a product. Consider the factor in the dashed box. We know the automorphism group of $(C_{n-k+1}, \{\alpha_k\})$ is $A_{2n-2k+1}$, then take a local deformation of the factor by a group element $g\in A_{2n-2k+1}$ that does not lie in $C_{n-k+1}$ (see Remark \ref{covering_pair}). This gives a local deformation of $X^J(w)$ in $X$ not from the action of $C_n$ and hence the Schubert variety is not homologically rigid.
\end{proof}
\begin{proof}[Proof of Theorem \ref{homological_subdiagram} when $G=C_n$]
Let $\alpha_k$ be the rightest marked simple root in $I$ on the Dynkin diagram $\Gamma$. If $k=1$, then $X=(C_n,\{\alpha_1\}) \cong (A_{2n-1}, \{\alpha_1\})$. In the following we assume that $k\geq 2$. Let $\pi: X=G/P_J\rightarrow G/P_K$ be the natural projection with $K=R\backslash \{\alpha_k\}$. Then the image $Y_0:=\pi(X_0)$ is again a Schubert variety of subdiagram type on $G/P_K$ and this is a Billey-Postnikov decomposition. If $(Y_0, G/P_K)$ is an exceptional pair, then from Proposition \ref{prop-exceptional-pair}, we are in the situation of Theorem 1.3 (7), which has been settled by Lemma \ref{Type_C_subdiagram_not_rigid}. If $(Y_0, G/P_K)$ is not an exceptional pair, note that the central fiber of $\pi$ is $P_K/P_J=G_K/(G_K\cap P_J)=G_{K'}/(G_{K'}\cap P_J)$ such that $G_{K'}$ is of type $A$, where $K'$ denotes the set of simple roots in $K$ lying to the left of $\alpha_k$. 
Then if we do further Billey-Postnikov decompositions, there is no exceptional pair by Proposition \ref{prop-exceptional-pair} and homological rigidity holds by Proposition \ref{prop homology rigidity reduction}.
\end{proof}

   \subsubsection{$F_4$-cases}
  
   The proof of Theorem \ref{homological_subdiagram} when $G=F_4$ follows from the following lemmas.
\begin{lemma}\label{F4_no_marked_root_out}
    When $G=F_4$, if $I\subset L$, $X_0=X^{J}(w)\subset X$ is homologically rigid except when we are in the situation (1),(2) in Theorem \ref{homological_subdiagram}, which are
    \begin{enumerate}[label=$\bullet$]
        \item  $X=(F_4,\{\alpha_3\}), L=\{\alpha_2,\alpha_3\}$ or $\{\alpha_3\}$;
        \item $X=(F_4,\{\alpha_4\}),  L=\{\alpha_3,\alpha_4\}$ or $\{\alpha_4\}$.
     \end{enumerate}
\end{lemma}
\begin{proof}
     When only one simple root is marked (i.e., $|I|=1$), we already know the non-rigid cases from \cite[Theorem 1.1]{HM13} (Cases (1),(2) in Theorem \ref{homological_subdiagram}). Now we consider the cases when two simple roots are marked (i.e., $|I|=2$). If $I=J^c=\{\alpha_1,\alpha_2\}$, no short root is marked and homological rigidity holds from Theorem \ref{main_thm_simple_ver}. If $I=J^c=\{\alpha_3, \alpha_4\}$, let $K=J\cup \{\alpha_4\}$, then take the Billey-Postnikov decomposition with respect to $(J,K)$ and no exceptional pair appears by Proposition \ref{prop-exceptional-pair}, since $\alpha_3,\alpha_4\in L$. 
     When only one root in $\{\alpha_3, \alpha_4\}$ is marked and $\alpha_2\in I, \alpha_1\notin I$, let $K=R\setminus \{\alpha_2\}$, then take the Billey-Postnikov decomposition with respect to $(J,K)$, in the base direction the marked root $\alpha_2$ is a long root and in the fiber direction it is of type $A_1\times A_2$, then the homological rigidity follows from Proposition \ref{prop homology rigidity reduction}. If only one root in $\{\alpha_3, \alpha_4\}$ is marked and $\alpha_2\notin I, \alpha_1\in I$,
     let $K=R\backslash\{\alpha_1\}$ and take the Billey-Postnikov decomposition with respect to $(J,K)$, then the fiber direction is of $C_3$ type, from Proposition \ref{prop-exceptional-pair}, the cases for which Proposition \ref{prop homology rigidity reduction} could not apply is 
     \begin{enumerate}
       \item[(a)] $I=L=\{\alpha_1,\alpha_3\}$;
       \item[(b)]  $I=L=\{\alpha_1, \alpha_4\}$;
       \item[(c)]  $I=\{\alpha_1, \alpha_4\}, L=\{\alpha_1, \alpha_3, \alpha_4\}$.
     \end{enumerate} All these cases will be shown to be homologically rigid by Lemma \ref{F4_special_case}.
     
     When we have three marked simple roots, $I\subset L$ forces $I=L$. If $\alpha_1\notin I $ or $\alpha_2\not\in I$, let $K=R\backslash \{\alpha_3, \alpha_4\}$; If $\alpha_3\notin I $ or $\alpha_4\not\in I$, let $K=R\backslash \{\alpha_1, \alpha_2\}$. Then we take the Billey-Postnikov decomposition with respect to $(J,K)$, the fiber direction is always $A_2$ type and exceptional pair does not appear. Now, in the base direction, the remaining discussion reduces to the cases where two simple roots are marked, one can observe that they do not belong to Cases (a), (b), (c) and have been settled.
\end{proof}
\begin{lemma}\label{F4_special_case}
    When $G=F_4$, for Cases (a)(b)(c) in Lemma \ref{F4_no_marked_root_out}, $X_0=X^{J}(w)\subset X$ is homologically rigid.
\end{lemma}
\begin{proof}
    As the proof in the Lemma \ref{type_B_subdiagram_special_rigid}. We prove the homological rigidity for these cases making use of both the statement and the proof of Proposition \ref{prop homology rigidity reduction}. Let $K=R\backslash \{\alpha_1\}$ and take the Billey-Postnikov decomposition $w=vu$ with respect to $(J,K)$. Then in the base level, there is no exceptional pair since $\alpha_1$ is a long root. 
 From the proof of Proposition \ref{prop homology rigidity reduction}, without loss of generality we may assume that the local deformation $\mathcal{M} \subset G/P_J\times \Delta$ induces a trivial deformation of $X^{K}(v)$ in the base level. Shrinking $\Delta$ around the origin, then for each $t\in \Delta$, $\mathcal{M}_t$ is still a product, we may write $\mathcal{M}_t=X^K(v)\times Y_t$ with $\mathcal{M}_0=X^{K}(v)\times X^J(u)$ and $Y_t\subset P_K/P_K\cap P_J$. We define \[ \gamma_t:Y_t\rightarrow \operatorname{Chow}(G/P_J)\]
by $\gamma(y)=[X^{K}(v)\times \{y\}]$ for $y\in Y_t$. From Lemma \ref{stab.subdiagram} and Corollary \ref{subdiagram_rigid_marked_root_outside}, by identifying $X^{K}(v)$ as a Schubert variety of subdiagram type in $X$ (which is homologically rigid), the Chow space of $X^{K}(v)$ in $X=G/P_J$ is $G/P_{R\backslash R_3}$, where $R_3=\{\alpha_2,\alpha_3\} $(Case (a)), $ \{\alpha_2,\alpha_4\}$ (Case (b),(c)). 

Now, $Y_t$ is a local deformation of $Y_0=X^{J}(u)$ in $G/P_{R\backslash R_3}$. By Lemma \ref{F4_marked_outside} below, $X^{J}(u)$ has homological rigidity in $G/P_{R\backslash R_3}$ and there is a holomorphic map $g:\Delta \rightarrow G$ such that $Y_t=g(t)\cdot Y_0$. Hence $\mathcal{M}_t=g(t)\cdot X_0$. Note that we have assume that the deformation at the base level is trivial, $g(t)$ stabilizes $X^{K}(v)\subset G/P_K$ and hence $g(t)\in 
P_{R\backslash \{\alpha_2\}}$. In particular for Case (c), $g(t)$ actually stabilizes $X_0 \subset X=G/P_J$, i.e., $\mathcal{M}$ is also a trivial familiy in this case assuming the induced family at the base level is trivial.
\end{proof}
\begin{lemma}\label{F4_marked_outside}
When $G=F_4$, if $I\backslash L\neq \emptyset$, $X_0=X^{J}(w)\subset X$ is homologically rigid except when we are in the situation (3)-(6) in Theorem \ref{homological_subdiagram}, which are 
 \begin{enumerate}[label=$\bullet$]
        \item$X=(F_4, \{\alpha_1,\alpha_3\}), L=\{\alpha_3\}$; 
        \item $X=(F_4, \{\alpha_1,\alpha_4\}), L=\{\alpha_4\}$ or $\{\alpha_3, \alpha_4\}$; 
        \item $X=(F_4,\{\alpha_1,\alpha_3, \alpha_4\}), L=\{\alpha_3\}$;
     \item $X=(F_4,\{\alpha_3, \alpha_4\}), L=\{\alpha_3\}$ or $\{\alpha_2,\alpha_3\}$.
      \end{enumerate}
\end{lemma}
\begin{proof}
If a simple root $\alpha\in I \backslash L$, then let $K=R\backslash\{\alpha\}$ and take the Billey-Postnikov decomposition with respect to $(J,K)$, then it suffices to check for the diagram $\Gamma_{R\backslash \{\alpha\}}$ corresponding to the fiber direction of the projection $G/P_J\rightarrow G/P_K$ (as $X^K(v)$ is a single point in this case). We can do it case by case.

 If $\alpha_1\in I\backslash L$, the diagram $\Gamma_{R\backslash \{\alpha_1\}}$ is exactly of  type $C_3$, then the non-rigid cases are either from the non-rigid cases for type  $C_3$(Lemma \ref{Type_C_subdiagram_not_rigid}), or the cases where $P_K/P_J$, the homogeneous subspace associated with $\Gamma_{R\backslash\{\alpha_1\}}(I\backslash \{\alpha_1\})$, has automorphism group larger than $C_3$ (Lemma \ref{lem-aut-larger}). They can be listed as
 \begin{enumerate}
     \item[(i)] $I=\{\alpha_1, \alpha_3\}, L=\{\alpha_3\}$;
     \item[(ii)] $I=\{\alpha_1, \alpha_4\}, L=\{\alpha_4\}$;
     \item[(iii)] $I=\{\alpha_1, \alpha_4\}, L=\{\alpha_3,\alpha_4\}$;
      \item[(iv)] $I=\{\alpha_1, \alpha_3, \alpha_4\}, L=\{\alpha_3\}$.
 \end{enumerate}
More precisely, Cases (i),(iv) are from our discussion on type $C$ cases in Section \ref{Cn_subdiagram} (Lemma \ref{Type_C_subdiagram_not_rigid}), Cases (ii),(iii) are from Lemma \ref{lem-aut-larger} that $P_K/P_J$ has automorphism group $A_5$ which gives deformations of $X_0$ not from $F_4$.

 If $\alpha_2\in I \backslash L$, then the diagram $\Gamma_{R\backslash \{\alpha_2\}}$ is of type $A_1\times A_2$ and every case is homologically rigid;  If $\alpha_3\in I \backslash L$, the diagram $\Gamma_{R\backslash \{\alpha_3\}}$ is of type $A_2\times A_1$, also every case is homologically rigid.
 
 If $\alpha_4\in I \backslash L$, the diagram $\Gamma_{R\backslash \{\alpha_4\}}$ is of type $B_3$,  then the non-rigid cases are either from the non-rigid cases for type $B_3$ (Lemma \ref{type_B_subdiagram_not_rigid}), or the cases where $P_K/P_J$, the homogeneous subspace associated with $\Gamma_{R\backslash\{\alpha_4\}}(I\backslash \{\alpha_4\})$, has automorphism group larger than $B_3$ (Lemma \ref{lem-aut-larger}). They can be listed as
 
 \begin{enumerate}
     \item[(i)] $I=\{\alpha_3, \alpha_4\}, L=\{\alpha_3\}$;
     \item[(ii)] $I=\{\alpha_3, \alpha_4\}, L=\{\alpha_2,\alpha_3\}$;
     \item[(iii)] $I=\{\alpha_1, \alpha_3, \alpha_4\}, L=\{\alpha_3\}$;
 \end{enumerate}
More precisely, Case (iii) is from our discussion on type $B$ cases in Section \ref{Bn_subdiagram} (Lemma \ref{type_B_subdiagram_not_rigid}), Cases (i),(ii) are from Lemma \ref{lem-aut-larger} that $P_K/P_J$ has automorphism group $D_4$ which gives deformations of $X_0$ not from $F_4$. In summary, we obtain Cases (3)-(6) in Theorem \ref{homological_subdiagram} which are non-rigid..
\end{proof}

\section{Schur rigidity in the subdiagram case}\label{Schur_higher_Pic}

In this section, we prove Theorem \ref{Schur_subdiagram} for the Schur rigidity of smooth Schubert varieties in the subdiagram case. We first consider the case when the subdiagram is connected, where the discussion will be divided into two parts (Proposition \ref{thm_with_condition_E}, Proposition \ref{no_condition_E}). In Proposition \ref{thm_with_condition_E}, we give a proof of Theorem \ref{Schur_subdiagram} when the subdiagram is connected, $X_0$ is of Picard number two and $G$ is of type $ADE$. In Proposition \ref{no_condition_E}, we give a proof of Theorem \ref{Schur_subdiagram} when the subdiagram is connected and satisfies certain conditions which always hold if $G$ is not of type $DE$, or $X_0$ is of Picard number at least three.
Then we finish the proof of Theorem \ref{Schur_subdiagram} for the general cases (Proposition \ref{product_case}).

Let $I=J^c \subset R $ be the set of marked roots and $L$ be the set of simple roots in the subdiagram. In this section we assume that \textbf{$I\cap L$ consists of long roots}. For $\alpha\in I$, we denote by $C_{\alpha}$ the set of simple roots in the connected component of the Dynkin subdiagram $\Gamma_{J\cup\{\alpha\}}$ that contains $\alpha$, which corresponds to the central fiber of the projection map $\pi_\alpha: X=G/P_J\rightarrow G/P_{J\cup\{\alpha\}}$ (see Section \ref{Mori_contraction_prelimi}).

\begin{lemma}\label{lem-intersection-roots-Dynkin}
Let $\beta_1,\beta_2$ be two distinct roots in $I$. Then $\Gamma_{C_{\beta_1}\cap C_{\beta_2}}$ is the unique connected component of $\Gamma_R$ connecting both $\beta_1$ and $\beta_2$.
\end{lemma}

\begin{proof}
Since $\Gamma_R$ contains no loops, there exists at most one connected component $\Gamma_{Q}$ of the subdiagram $\Gamma_J$ connecting both $\beta_1$ and $\beta_2$ in $\Gamma_R$. Write $Q=\emptyset$ when this component does not exist. By the definition of $C_{\beta_1}$ and $C_{\beta_2}$, we have $Q\subset C_{\beta_1}\cap C_{\beta_2}$. The non-existence of loops also implies that $C_{\beta_1}\cap C_{\beta_2} \subset Q$, hence $C_{\beta_1}\cap C_{\beta_2} = Q$.  
\end{proof}

\begin{proposition}\label{prop-intersection-two-roots}
 Let $X_0=G_L/P_{L\cap J}\subset X=G/P_J$ be a smooth Schubert variety associated with the connected subdiagram $\Gamma_{L}(I\cap L)$ of the connected marked Dynkin diagram $\Gamma_R$. Suppose $|I\cap L|\geq 2$ (i.e., $X_0$ is of at least Picard number 2). Then
\begin{eqnarray}\label{eqn-intersection-Dynkin}
\bigcap_{\alpha \in I\cap L} C_{\alpha}\subset L,
\end{eqnarray} 
unless all the following conditions hold:
\begin{itemize}
\item[(i)] $G$ is of type $D$ or $E$, and $|I\cap L|=2$;
\item[(ii)] the unique simple root $\gamma\in R$ with $|N(\gamma)|=3$ satisfies $\gamma\in L\cap J$;
\item[(iii)]  up to reorder of the connected components $\Gamma_{K_1},\Gamma_{K_2},\Gamma_{K_3}$ of $\Gamma_{R\backslash\{\gamma\}}$, we have  $|I\cap L\cap K_1|=|I\cap L\cap K_2|=1$, $I\cap L\cap K_3=\emptyset$, and the single element set $K_3\cap N(\gamma)\subset L\cap J$.
\end{itemize} 
In particular, \eqref{eqn-intersection-Dynkin} holds when $|I\cap L|\geq 3$.
\end{proposition}

\begin{proof}
Suppose \eqref{eqn-intersection-Dynkin} fails. For any two distinct roots $\beta_1, \beta_2\in I\cap L$, we have $C_{\beta_1}\cap C_{\beta_2}\not\subset L$. By Lemma \ref{lem-intersection-roots-Dynkin}, $\Gamma_{Q_{1, 2}}$ is the unique connected component of $\Gamma_J$ connecting both $\beta_1$ and $\beta_2$, where $Q_{1, 2}:=C_{\beta_1}\cap C_{\beta_2}$. Now we set $S_{1, 2}:=Q_{1, 2}\cup\{\beta_1, \beta_2\}$. We can conclude that $\Gamma_{S_{1, 2}}$ is of type $D$ or $E$, since in the case of other types of connected Dynkin diagrams, the fact $\beta_1,\beta_2\in L$ would imply that the path $\Gamma_{Q_{1, 2}}$ connecting them is also contained in $\Gamma_L$. Furthermore, $G$ is of type $D$ or $E$, and $\gamma\in Q_{1, 2}\subset L\cap J$. Up to the reorder of $K_1, K_2$ and $K_3$, we have $\beta_i\in I\cap L\cap K_i$ for $i=1,2$. 
Besides, the single element set $K_3\cap N(\gamma)\subset L\cap J$ since otherwise $\Gamma_{Q_{1, 2}}$ is contained in $\Gamma_L$.

It remains to show $|I\cap L|=2$, since in this case conclusions (i) (ii) and (iii) follow immediately. Suppose $|I\cap L|\geq 3$ and take $\beta_3\in I\cap L$ different from $\beta_1$ and $\beta_2$. Then the argument above implies that for all $1\leq i<j\leq 3$, the root $\gamma\in Q_{i, j}\subset L\cap J$ and the Dynkin diagram $\Gamma_{S_{i, j}}$ is of type $D$ or $E$, where $Q_{i, j}:=C_{\beta_i}\cap C_{\beta_j}$ and $S_{i, j}:=Q_{i, j}\cup\{\beta_i, \beta_j\}$. 

As a consequence, for each $k=1,2,3$, the set $K_k\cap\{\beta_1,\beta_2,\beta_3\}$ has at most one element. Since we have fixed the subscript of $K_k$ such that $\beta_1\in K_1$ and $\beta_2\in K_2$, it must hold that $\beta_3\in K_3$. 

By the characteristic of Dynkin diagrams, there is at least one connected component $\Gamma_{K_k}$ of $\Gamma_{R\setminus\{\gamma\}}$ that consists of a unique node. Without loss of generality, we may assume that $\Gamma_{K_3}$ is such a component. Then $K_3=\{\beta_3\}\subset I$, and the Dynkin diagram $\Gamma_{R\setminus K_3}$ is of type $A$. As a subdiagram of $\Gamma_{R\setminus K_3}$, we know that $\Gamma_{S_{1, 2}}$ is also of type $A$. It is a contradiction, and thus $|I\cap L|=2$.
\end{proof}



\begin{example}\label{Exam_of_condition_E}

  When $G$ is $D_6$, $\Gamma_{L}(I\cap L)$ is given in the box in the following Dynkin diagram, $\gamma \in L\cap J, I\cap L=\{\alpha,\beta\}$, $|N(\gamma) |=3$ and it is easy to check that it satisfies all conditions listed in Proposition \ref{prop-intersection-two-roots} and $C_\beta\cap C_\alpha \not \subset L$.
    \begin{center}
     \begin{tikzpicture}[scale=0.5]
  \draw[thick] (0,0) -- (4.5, 0)  ; 
  \draw[thick] (4.5,0) -- (5.5, 1)  ;
  \draw[thick] (4.5,0) -- (5.5, -1)  ;
\draw[ thick, fill=white] (0,0) circle (3pt) node[above, outer sep=3pt]{};
		\draw[ thick, fill=black] (1.5,0) circle (3pt) node[above, outer sep=3pt]{$\beta$};
        \draw[ thick, fill=white] (3,0) circle (3pt) node[above, outer sep=3pt]{};
        \draw[ thick, fill=white] (4.5,0) circle (3pt) node[above, outer sep=3pt]{$\gamma$};
         \draw[ thick, fill=white] (5.5,1) circle (3pt) node[above, outer sep=3pt]{};
        \draw[ thick, fill=black] (5.5,-1) circle (3pt) node[above, outer sep=3pt]{$\alpha$};
         \draw[thick] (1.1,0.6) -- (4.5,0.6)  ;
           \draw[thick] (1.1,-0.6) -- (4.6,-0.6)  ;
           \draw[thick] (1.1,0.6) -- (1.1,-0.6);
            \draw[thick] (4.5,-0.6) -- (5.5,-1.6);
            \draw[thick] (5.5,-1.6) -- (6.2,-1);
                 \draw[thick] (4.6,0.6) -- (6.2,-1);
	\end{tikzpicture} 	
     \end{center}
\end{example}

Now we start to prove Theorem \ref{Schur_subdiagram} when $\Gamma_L(I\cap L)$ is connected. We first deal with the case when $|I\cap L|=1$ (Lemma \ref{schur_pic_num_one}), and then the case when $|I\cap L|=2$ and $G$ is of type $ADE$ (Proposition \ref{thm_with_condition_E}). Finally, we deal with the case when $|I\cap L|\geq 2$ and \eqref{eqn-intersection-Dynkin} holds (Proposition \ref{no_condition_E}). By Proposition \ref{prop-intersection-two-roots}, these covers all possibilities.
Before the discussion of these three cases, let us make some preparations.

\begin{lemma}\label{lem-fiber-bundle-Schubert}
Let $X_0=G_L/(G_L\cap P_{J})$ be a Schubert variety of subdiagram type on the rational homogeneous space $X=G/P_{J}$. Then $X_0$ is a fiber bundle over the projective space if and only if there exists $\alpha\in I\cap L$ such that $G_L/(G_L\cap P_{R\backslash\{\alpha\}})$ is isomorphic to a projective space.
\end{lemma}

\begin{proof}
The set of Mori contractions on $X_0$ coincides with the set of natural projetions $X_0=G_L/(G_L\cap P_{J})\rightarrow G_L/(G_L\cap P_K)$ with $L\cap J\subset K\subset L$. The Picard number of $G_L/(G_L\cap P_K)$ is $|L\backslash K|$. Then the conclusion follows.
\end{proof}

\begin{lemma}\label{rigidity_from_big_to_small}
    Let $X_1=G_1/P_1\subset X_0=G_0/P_0\subset X=G/P_J$ be smooth Schubert varieties associated with Dynkin subdiagrams $\Gamma_{L}\subset \Gamma_0 \subset\Gamma(I)$ and the marked roots in $I\cap L$ are all long roots. If $X_1$ is Schur rigid (resp. homologically rigid) in $X$ , then $X_1$ is Schur rigid (resp. homologically rigid) in $X_0$.
\end{lemma}
\begin{proof}
 Since $I\cap L$ consists of long roots, from Lemma \ref{lem-aut-larger}, $G=\mathrm{Aut}_0(X)$ and $G_0=\mathrm{Aut}_0(X_0)$ in any case. Suppose that $X_1$ is Schur rigid in $X$ and $Z\subset X_0$ is a subvariety with the homology class $[Z]=r[X_1]$, then $Z=g_1\cdot X_1+\cdots+g_r\cdot X_1$ for some $g_i\in G$. It suffices to show that for each $1\leq i\leq r$, there exists $g'_i\in G_0$ such that $g_i\cdot X_1=g'_i\cdot X_1$. Note that for each $1\leq i\leq r$, $g_i\cdot X_1$ is homologous to $X_1$ in $X_0$ (since $H_*(X_0,\mathbb{Z})\rightarrow H_*(X,\mathbb{Z})$ is an inclusion), then it suffices to prove the homological rigidity of $X_1$ in $X_0$. Now the lemma is immediate from Corollary \ref{subdiagram_rigid_marked_root_outside}, as $I\cap L$ consists of long roots. 
\end{proof}

\begin{lemma}\label{schur_pic_num_one}
Assume that the subdigram $\Gamma_L$ is connected, $I\cap J$ consists of long roots, $|I\cap L|=1$ and $X_0=G_L/(G_L\cap P_{J})$ is not isomorphic to a projective space. Then the Schubert variety $X_0=G_L/(G_L\cap P_{J})$ on $X=G/P_{J}$ has Schur rigidity.
\end{lemma}
\begin{proof}
  Write $I\cap L=\{\alpha\}$. Since $X_0=G_L/(G_L\cap P_{J})$ is of Picard number one, we have $L\subset C_\alpha$. By the assumption that $X_0$ is not isomorphic to a projective space, we know  $X_0\subset  G_{C_\alpha}/P_{C_\alpha\backslash\{\alpha\}}$ is Schur rigid, from Theorem \ref{Schur_HM}. Note that $\alpha$ is a long root and no exceptional pair appears, the rest can be deduced easily by repeating the proof of Lemma \ref{pic_num_one}.
\end{proof}

\begin{lemma}\label{intersection_max}
Let $G$ be of type $ADE$, and let
$X_1=G_{L_1}/P_{R\backslash I\cap L_1}, X_2=G_{L_2}/P_{R\backslash I\cap L_2}\subset X=G/P_J$ be smooth Schubert varieties associated with the subdiagrams $\Gamma_{L_1}(I\cap L_1), \Gamma_{L_2}(I\cap L_2)$ of the marked Dynkin diagram $\Gamma_R(I)$ respectively, and $L_1, L_2$ has non-trivial intersection. Then for $g_1,g_2 \in G$,
\begin{eqnarray}\label{eqn-intersection-dimension}
\dim(g_1\cdot X_1 \cap g_2 \cdot X_2 )\leq \dim(X_3)
\end{eqnarray}
where $X_3=G_{L_1\cap L_2}/G_{L_1\cap L_2}\cap P_{R\backslash I\cap L_1\cap L_2}$. Moreover the equality holds if and only if there exists $g_3\in G$ such that $g_1\cdot X_1\cap g_2\cdot X_2=g_3\cdot X_3.$
\end{lemma}

\begin{proof}
   Without loss of generality, we may assume that $g_1=e$. Let $P_{A_2}$ be the stabilizer of $X_2$, we have the following projections.
   \begin{eqnarray*}
\xymatrix{& G/P_{J\cap A_2}\ar[ld]_{\pi_1}\ar[rd]^{\pi_2}  & \\
G/P_{J}  & &G/P_{A_2}  \\}
\end{eqnarray*}

Let $\phi: \pi_1^{-1}(X_1)\rightarrow G/P_{A_2}$ be the restrition of $\pi_2$, and it is a $B$-equivariant morphism. Then the loci $M_k:=\{y\in G/P_{A_2}\mid \dim\phi^{-1}(y)\geq k\}$, $k\geq 0$, are all $B$-stable closed  subsets, hence they contain the base point $q$ of $G/P_{A_2}$. It follows that 
\begin{eqnarray}\label{eqn-dim-inequality}
\dim \phi^{-1}(g_2\cdot q) \leq \dim \phi^{-1}(q) \mbox{ for all } g_2\in G.
\end{eqnarray}

Note that $\pi_1$ sends each fiber of $\phi$ isomorphically onto its image in $G/P_J$. Then by \eqref{eqn-dim-inequality} the dimension of $\pi_1(\phi^{-1}(q))=X_3$ is larger than or equal to $\pi_1(\phi^{-1}(g_2\cdot q))=X_1\cap g_2X_2$ for all $g_2\in G$. This verifies the inequality \eqref{eqn-intersection-dimension}.

Set $d:=\dim X_3=\dim\phi^{-1}(q)$. Then $M_d$ is $B$-stable, and $\dim\phi^{-1}(y)=d$ for each $y\in M_d$. Each irreducible component of $M_d$ is a Schubert variety on $G/P_{A_2}$, and contains the base point $q$. 

Take any $y\in M_d$ near the base point $q\in G/P_{A_2}$.  We can find an affine curve $C\subset G$ such that the neutral element $e\in C$ and the point $y=a_0\cdot q$ for some $a_0\in C$. Now we get a flat family $\phi^{-1}(a\cdot q)$ with $a\in C$. Since $\pi_1$ sends $\phi^{-1}(a\cdot q)$ isomorphically onto $X_1\cap a\cdot X_2$ for each $a\in C$. 
Since the intersection $X_1\cap X_2$ is smooth and connected, so is $X_1\cap aX_2$ for each point $a$ in an open neighborhood $C^o$ of $e\in C$.
By the assumption of $y$, we may assume $a_0\in C^o$.
Now for each $a\in C^o$, the intersection $X_1\cap a\cdot X_2$ has the same homology class with $X_3=X_1\cap X_2$. By Theorem \ref{main_thm_simple_ver}, since $G$ is of type $ADE$, for each $a\in C^o$ there exists $g\in G$ such that $X_1\cap a \cdot X_2=g \cdot X_3=g \cdot X_1\cap g\cdot X_2$. In particular, there exists $g_0\in G$ such that $\pi_1(\phi^{-1}(y))=X_1\cap a_0 \cdot X_2=g_0\cdot X_1\cap g_0\cdot X_2$. 

Since $X_3$ is a Schubert variety on $G/P_J$, its stabilizer in $G$ is a parabolic subgroup. Hence the set of $G$-translations of $X_3$ is a projective family, and thus the set  $F:=\{y\in M_d\mid \pi_1(\phi^{-1}(y))=g\cdot X_1\cap g\cdot X_2 \mbox{ for some } g\in G \}$ is projective. The argument in the last paragraph shows that $F$ contains a nonempty open subset of each irreducible component of $M_d$. It follows that $F=M_d$. This completes the proof of this proposition.
\end{proof}

\begin{lemma}\label{intersection}
    Suppose that $I\cap L=\{\alpha, \beta\}$, then $X_0=G_L/G_L\cap P_{R\backslash \{\alpha, \beta\}}=\pi^{-1}_\alpha(G_L/G_L\cap P_{R\backslash \{\beta\}})\cap \pi^{-1}_\beta(G_L/G_L\cap P_{R\backslash\{\alpha\}})$.
\end{lemma}
\begin{proof} 
Without ambiguity, we will use the same notation for the homogeneous subspaces if they are mapped isomorphically under the projections $\pi_\alpha, \pi_\beta$. The following diagram is for the projections. For more illustrations, see Example \ref{fiber_intersection}.
\begin{eqnarray*}
\xymatrix{& G/P_{J}\ar[ld]_{\pi_\alpha}\ar[rd]^{\pi_\beta}  & \\
G/P_{J\cup \{\alpha\}}  & &G/P_{J\cup \{\beta\}}  \\}
\end{eqnarray*}
     Note that for the base point $o\in G_L/G_L\cap P_{R\backslash\{\alpha\}}\subset G/P_{J\cup \{\beta\}}$, $\pi_\beta^{-1}(o)=G_{C_\beta}/G_{C_\beta}\cap P_{R\backslash \{\beta\}}\subset G/P_J $ is mapped biholomorphically into $G/P_{J\cup \{\alpha\}}$ through the projection $\pi_\alpha$. Then 
     \[\begin{aligned}
        \pi_\alpha\circ \pi^{-1}_\beta(o) \cap (G_L/G_L\cap P_{R\backslash \{\beta\}})&=(G_{C_\beta}/G_{C_\beta}\cap P_{R\backslash \{\beta\}}) \cap (G_L/G_L\cap P_{R\backslash \{\beta\}})\\&=G_{L\cap C_\beta}/(G_{L\cap C_\beta}\cap P_{R\backslash \{\beta\}}). \end{aligned} \]
        Now, since $\pi_\alpha$ is an isomorphism on $\pi_\beta^{-1}(o)$, we know \[\pi^{-1}_\beta(o) \cap \pi^{-1}_\alpha(G_L/G_L\cap P_{R\backslash \{\beta\}})=G_{L\cap C_\beta}/G_{L\cap C_\beta}\cap P_{R\backslash \{\beta\}},\] which is exactly the central fiber of $\pi_\beta|_{X_0}:X_0 \rightarrow G_L/G_L\cap P_{R\backslash\{\alpha\}}$. By varying $o$ through $G_L$-actions, we prove the lemma.
\end{proof}

\begin{example}\label{fiber_intersection}
To give more illustrations, we write the example as follows, where $L$ and $C_\beta$ are the sets of simple roots contained in the box and the dashed box repectively.
 \begin{figure}[htbp]
\begin{tikzcd}[column sep=0.3cm]
\node (node1){}; &
\node(node2) {
  \begin{tikzpicture}[scale=0.5] 
  \draw[thick] (0,0) -- (4.5, 0)  ; 
  \draw[thick] (4.5,0) -- (5.5, 1)  ;
  \draw[thick] (4.5,0) -- (5.5, -1)  ;
\draw[ thick, fill=white] (0,0) circle (3pt) node[above, outer sep=3pt]{};
		\draw[ thick, fill=black] (1.5,0) circle (3pt) node[above, outer sep=3pt]{$\beta$};
        \draw[ thick, fill=white] (3,0) circle (3pt) node[above, outer sep=3pt]{};
        \draw[ thick, fill=white] (4.5,0) circle (3pt) node[above, outer sep=3pt]{};
         \draw[ thick, fill=white] (5.5,1) circle (3pt) node[above, outer sep=3pt]{};
        \draw[ thick, fill=black] (5.5,-1) circle (3pt) node[above, outer sep=3pt]{$\alpha$};
         \draw[thick] (-0.4,0.6) -- (4.7,0.6)  ;
           \draw[thick] (-0.4,-0.6) -- (4.5,-0.6)  ;
           \draw[thick] (-0.4,0.6) -- (-0.4,-0.6);
            \draw[thick] (4.5,-0.6) -- (5.5,-1.6);
            \draw[thick] (5.5,-1.6) -- (6.2,-1);
             \draw[thick] (4.7,0.6) -- (6.2,-1)  ;

               \draw[thick,dashed] (-0.6,0.8) -- (4.6,0.8)  ;
           \draw[thick,dashed] (-0.6,-0.8) -- (4.6,-0.8)  ;
           \draw[thick,dashed] (-0.6,0.8) -- (-0.6,-0.8);
            \draw[thick,dashed] (4.6,0.8) -- (5.5,1.6);
            \draw[thick,dashed] (5.5,1.6) -- (6.2,0.9);
             \draw[thick,dashed] (4.5,-0.6) -- (6.2,0.9)  ;
	\end{tikzpicture} };
       \arrow[ld, "\pi_{\alpha}", from=node2, to=node4]   \arrow[rd, "\pi_\beta", from=node2, to=node6] & 
       \node(node3){}; \\
      \node (node4) {
        \begin{tikzpicture}[scale=0.5]
            \draw[thick] (0,0) -- (4.5, 0)  ; 
  \draw[thick] (4.5,0) -- (5.5, 1)  ;
  \draw[thick] (4.5,0) -- (5.5, -1)  ;
\draw[ thick, fill=white] (0,0) circle (3pt) node[above, outer sep=3pt]{};
		\draw[ thick, fill=black] (1.5,0) circle (3pt) node[above, outer sep=3pt]{$\beta$};
        \draw[ thick, fill=white] (3,0) circle (3pt) node[above, outer sep=3pt]{};
        \draw[ thick, fill=white] (4.5,0) circle (3pt) node[above, outer sep=3pt]{};
         \draw[ thick, fill=white] (5.5,1) circle (3pt) node[above, outer sep=3pt]{};
        \draw[ thick, fill=white] (5.5,-1) circle (3pt) node[above, outer sep=3pt]{$\alpha$};
         \draw[thick] (-0.4,0.6) -- (4.7,0.6)  ;
           \draw[thick] (-0.4,-0.6) -- (4.5,-0.6)  ;
           \draw[thick] (-0.4,0.6) -- (-0.4,-0.6);
            \draw[thick] (4.5,-0.6) -- (5.5,-1.6);
            \draw[thick] (5.5,-1.6) -- (6.2,-1);
             \draw[thick] (4.7,0.6) -- (6.2,-1)  ;

               \draw[thick,dashed] (-0.6,0.8) -- (4.6,0.8)  ;
           \draw[thick,dashed] (-0.6,-0.8) -- (4.6,-0.8)  ;
           \draw[thick,dashed] (-0.6,0.8) -- (-0.6,-0.8);
            \draw[thick,dashed] (4.6,0.8) -- (5.5,1.6);
            \draw[thick,dashed] (5.5,1.6) -- (6.2,0.9);
             \draw[thick,dashed] (4.5,-0.6) -- (6.2,0.9)  ;
	\end{tikzpicture} };
    & \node[xshift=3cm](node6) { \begin{tikzpicture}[scale=0.5]
            \draw[thick] (0,0) -- (4.5, 0)  ; 
  \draw[thick] (4.5,0) -- (5.5, 1)  ;
  \draw[thick] (4.5,0) -- (5.5, -1)  ;
\draw[ thick, fill=white] (0,0) circle (3pt) node[above, outer sep=3pt]{};
		\draw[ thick, fill=white] (1.5,0) circle (3pt) node[above, outer sep=3pt]{$\beta$};
        \draw[ thick, fill=white] (3,0) circle (3pt) node[above, outer sep=3pt]{};
        \draw[ thick, fill=white] (4.5,0) circle (3pt) node[above, outer sep=3pt]{};
         \draw[ thick, fill=white] (5.5,1) circle (3pt) node[above, outer sep=3pt]{};
        \draw[ thick, fill=black] (5.5,-1) circle (3pt) node[above, outer sep=3pt]{$\alpha$};
         \draw[thick] (-0.4,0.6) -- (4.7,0.6)  ;
           \draw[thick] (-0.4,-0.6) -- (4.5,-0.6)  ;
           \draw[thick] (-0.4,0.6) -- (-0.4,-0.6);
            \draw[thick] (4.5,-0.6) -- (5.5,-1.6);
            \draw[thick] (5.5,-1.6) -- (6.2,-1);
             \draw[thick] (4.7,0.6) -- (6.2,-1)  ;
	\end{tikzpicture}};
   \end{tikzcd}
\end{figure}
\end{example}

\begin{proposition}\label{thm_with_condition_E}
     Assume that $G$ is of type $ADE$, the subdigram $\Gamma_L(I\cap L)$ is connected and $|I\cap L|=2$ (i.e., $X_0$ is of Picard number 2), then Theorem \ref{Schur_subdiagram} holds.
\end{proposition}
    \begin{proof}
 Let $Z$ be an irreducible reduced subvariety in $X$ with $[Z]=r[G_L/G_L\cap P_{J}]$.  Write $I\cap L=\{\alpha, \beta\}$. We still have the projections $\pi_\alpha, \pi_\beta$. 
 \begin{eqnarray*}
\xymatrix{& G/P_{J}\ar[ld]_{\pi_\alpha}\ar[rd]^{\pi_\beta}  & \\
G/P_{J\cup \{\alpha\}}  & &G/P_{J\cup \{\beta\}}  \\}
\end{eqnarray*}
We know $[\pi_\alpha(Z)]\in \mathbb{R}^+[G_L/G_L\cap P_{R\backslash\{\beta\}}] .$
 By Lemma \ref{lem-fiber-bundle-Schubert}, $G_L/(G_L\cap P_{R\backslash \{\beta\}})$ is not a projective space under the assumption of Theorem \ref{Schur_subdiagram}.
Thus $(G_L/G_L\cap P_{R\backslash \{\beta\}}, G/P_{J\cup\{\alpha\}})$ is Schur rigid, by Lemma \ref{schur_pic_num_one}. 
As $\pi_\alpha(Z)$ is irreducible and reduced, we have $\pi_{\alpha}(Z)=g_1\cdot (G_L/G_L\cap P_{R\backslash \{\beta\}})$ for some $g_1\in G$ (as $G$ is of type $ADE$, there is no exceptional pair). Similarly $\pi_\beta(Z)=g_2\cdot (G_L/G_L\cap P_{R\backslash \{\alpha\}})$ for some $g_2\in G$.  Without loss of generality we may assume that $g_1=e$, and we will show that $Z=X_0$ under this assumption.
   We know \begin{equation}
       Z\subset \pi^{-1}_\alpha(G_L/G_L\cap P_{R\backslash \{\beta\}})\cap g_2 \cdot\pi^{-1}_\beta(G_L/G_L\cap P_{R\backslash\{\alpha\}}). \end{equation}  
  For any $x\in g_2\cdot (G_L/G_L\cap P_{R\backslash \{\alpha\}})$, let $Z_x:=Z\cap \pi^{-1}_\beta(x)$, then 
  \begin{equation}\label{Z_x_contained}
  Z_x\subset \pi^{-1}_\alpha(G_L/G_L\cap P_{R\backslash \{\beta\}})\cap \pi^{-1}_\beta(x),\end{equation}     
where $\pi_\beta^{-1}(x)=g_{2,x}\cdot\pi_\beta^{-1}(o)=g_{2,x}\cdot G_{C_\beta}/(G_{C_\beta}\cap P_{R\backslash \{\beta\}})$ for some $g_{2,x}\in G$ (depending on $x$). By the upper semi-continuity of the fiber dimension,
\begin{equation}\label{Z_x_dim}
\begin{aligned}
    \dim(Z_x)&\geq \dim(Z)-\dim(G_L/G_L\cap P_{R\backslash \{\alpha\}}) \\
    & =\dim(G_L/G_L\cap P_J)-\dim(G_L/G_L\cap P_{R\backslash \{\alpha\}})\\&
    =\dim(G_{L\cap C_\beta}/G_{L\cap C_\beta}\cap P_{R\backslash \{\beta\}}).
\end{aligned}
\end{equation}
Furthermore, in (\ref{Z_x_contained}), we take the projection $\pi_\alpha$ on $Z_x$, then  \begin{equation}\label{pi_alpha_Z_x}
    \begin{aligned}
    \pi_{\alpha} (Z_x)&\subset  (G_L/G_L\cap P_{R\backslash \{\beta\}})\cap  \pi_\alpha\circ \pi^{-1}_\beta(x)\\
    &=(G_L/G_L\cap P_{R\backslash \{\beta\}})\cap g_{2,x}\cdot \pi_\alpha(G_{C_\beta}/G_{C_\beta}\cap P_{R\backslash\{\beta\}}). 
    \end{aligned}\end{equation}
Note that $\pi_\alpha$ sends $\pi_\beta^{-1}(x)$ (and hence also $Z_x$) isomorphically onto its image, we have $\dim(\pi_{\alpha}(Z_x))=\dim(Z_x)$ and rewrite (\ref{pi_alpha_Z_x}) as
 \begin{equation}\label{pi_alpha_Z_x_rewrite}
    \pi_{\alpha} (Z_x)\subset 
    (G_L/G_L\cap P_{R\backslash \{\beta\}})\cap g_{2,x}\cdot (G_{C_\beta}/G_{C_\beta}\cap P_{R\backslash\{\beta\}}) 
  \end{equation} by identifying $G_{C_\beta}/(G_{C_\beta}\cap P_{R\backslash\{\beta\}})$ as a homogeneous subspace in $G/P_{J\cup\{\alpha\}}$.
 From (\ref{Z_x_dim}), we have
\begin{equation}
    \dim(\pi_\alpha(Z_x))=\dim(Z_x)\geq \dim(G_{L\cap C_\beta}/G_{L\cap C_\beta}\cap P_{R\backslash\{\beta\}}).\end{equation}
We know from Lemma \ref{intersection_max} that for each $x\in G_L/(G_L\cap P_{R\backslash\{\alpha\}})$, there exists $g_{3,x}\in G$ (depending on $x$) such that 
\begin{equation}
 \pi_\alpha(Z_x)
    =g_{3,x}\cdot (G_{L\cap C_\beta}/G_{L\cap C_\beta}\cap P_{R\backslash \{\beta\}})\subset G_L/(G_L\cap P_{R\backslash\{\beta\}}).
   \end{equation}
     Note that $\alpha$ is adjacent to $C_\beta$ in $L$, we know from Lemma \ref{stab.subdiagram} that $\alpha$ is the only simple root in $L$ so that the parabolic subgroup $P_{R\backslash\{\alpha\}}$ stabilizes $G_{L\cap C_\beta}/(G_{L\cap C_\beta}\cap P_{R\backslash \{\beta\}})$. Then consider the Chow space of \[G_{L\cap C_\beta}/(G_{L\cap C_\beta}\cap P_{R\backslash \{\beta\}})\subset G_L/(G_L\cap P_{R\backslash\{\beta\}}),\] it is exactly isomorphic to $G_L/(G_L\cap P_{R\backslash\{\alpha\}})$, since $\beta$ is a long root.  
Therefore, we have an injective (and hence bijective) holomorphic map \[\begin{aligned}
    x\in g_2\cdot (G_L/G_L\cap P_{R\backslash\{\alpha\}}) \rightarrow &[g_{3,x}\cdot G_{L\cap C_\beta}/G_{L\cap C_\beta}\cap P_{R\backslash\{\beta\}}] \\&\in \operatorname{Chow}(G_L/G_L\cap P_{R\backslash\{\beta\}}) ,\end{aligned}\] and then $Z=\mathcal{U}=X_0 \subset X$,
      where $\mathcal{U}$ is the universal family of the Chow space of $G_{L\cap C_\beta}/(G_{L\cap C_\beta}\cap P_{R\backslash \{\beta\}})\subset G_L/(G_L\cap P_{R\backslash\{\beta\}})$ (also we must have $g_2\cdot (G_L/G_L\cap P_{R\backslash\{\alpha\}} )=G_L/G_L\cap P_{R\backslash\{\alpha\}}$). 
    \end{proof}

\begin{proposition}\label{no_condition_E}
Suppose that $\Gamma_L$ is connected, $I\cap L$ consists of long roots and $|I\cap L|\geq 2$, $\bigcap\limits_{\alpha\in I\cap L} C_\alpha\subset L$ and $X_0=G_L/(G_L\cap P_{J})$ is not a fiber bundle over the projective space. Then the Schubert variety $X_0$ on $X=G/P_J$ has Schur rigidity.
\end{proposition}

\begin{proof}
Denote by $I\cap L=\{\beta_1,\beta_2,\ldots,\beta_m\}$, all of which are long roots. Consider the natural projections $\phi_k: X=G/P_J\rightarrow G/P_{R\backslash\{\beta_k\}}$ for each $k$. Then $\phi_k(X_0)$ is a Schubert variety of subdiagram type on $G/P_{R\backslash\{\beta_k\}}$, it is not isomorphic to a projective space by Lemma \ref{lem-fiber-bundle-Schubert}, and it has Schur rigidity in $G/P_{R\backslash\{\beta_k\}}$ by Theorem \ref{Schur_HM}.

Let $Z$ be an irreducible reduced closed subvariety of $X$ with homology class $[Z]=r[G_L/G_L\cap P_J]$ for some $r\geq 1$. Now $\phi_1(Z)$ is an irreducible closed subvariety (equipped with the induced reduced structure) on $G/P_{R\backslash\{\beta_1\}}$ whose homology class $[\phi_1(Z)]\in\mathbb{R}^+[\phi_1(X_0)]$. Since $\phi_1(X_0)$ has Schur rigidity in $G/P_{R\setminus\{\beta_1\}}$ (by Theorem \ref{Schur_HM}), we know that there exists $g_1\in G$ such that $g_1\cdot \phi_1(Z)=\phi_1(X_0)$.

The projective subvariety $Z_1:=g_1\cdot Z$ of $X$ is contained in $\phi_1^{-1}\phi_1(X_0)\subset G_{L_1}/(G_{L_1}\cap P_{L_1\backslash I})$, where $L_1:=L\cup C_{\beta_1}$. As a consequence, $\phi_k(Z_1)$ is contained in the Schubert variety $G_{L_1}/(G_{L_1}\cap P_{L_1\backslash\{\beta_k\}})$ of $G/P_{R\backslash\{\beta_k\}}$ for all $k=1,\ldots,m$. Then $\phi_k(Z_1)$ has Schur rigidity in $G_{L_1}/(G_{L_1}\cap P_{L_1\backslash\{\beta_k\}})$ by Lemma \ref{rigidity_from_big_to_small}.

The inclusion $G_{L_1}/(G_{L_1}\cap P_{L_1\backslash\{\beta_2\}})\subset G/P_{R\backslash\{\beta_2\}}$ induces the injection of homology groups. Then $\phi_2(Z_1)$ is an irreducible closed subvariety (equipped with the induced reduced structure) on $G_{L_1}/(G_{L_1}\cap P_{R\backslash\{\beta_2\}})$ whose homology class $[\phi_2(Z_1)]\in\mathbb{R}^+[\phi_2(X_0)]$. Since $\phi_2(X_0)$ has Schur rigidity in $G_{L_1}/(G_{L_1}\cap P_{R\backslash\{\beta_2\}})$, we know that there exists $g_2\in G_{L_1}$ such that $g_2\cdot\phi_2(Z_1)=\phi_2(X_0)$.

Now, the projective subvariety $Z_2:=g_2\cdot Z_1$ of $X$ is contained in \[(G_{L_1}/G_{L_1}\cap P_{L_1\backslash I}) \cap \phi^{-1}_2\phi_2(X_0)\subset G_{L_2}/(G_{L_2}\cap P_{L_2\backslash I}),\] where $L_2:=L_1\cap (L\cup C_{\beta_2})=\bigcap_{i=1,2} (L\cup C_{\beta_i})$.

By induction on $j=1,2,\ldots, m-1$, we can choose $g_{j+1}\in G_{L_j}$ with $L_j:=\bigcap_{i=1}^j (L\cup C_{\beta_i})$, and set $Z_{j+1}:=g_{j+1}\cdot Z_j$ such that $\phi_{j+1}(Z_{j+1})=\phi_{j+1}(X_0)$. It follows that 
$Z_{j+1}\subset G_{L_{j+1}}/(G_{L_{j+1}}\cap P_{L_{j+1}\backslash I}).$
In particular, the variety $Z_m=(g_m g_{m-1}\cdots g_1)\cdot Z$ is contained in $G_{L_m}/(G_{L_m}\cap P_{L_m\backslash I})=X_0$, by the assumption that $\bigcap\limits_{\alpha\in I\cap L} C_\alpha\subset L$. Since $X_0$ is irreducible with $\dim X_0=\dim Z$, we have $Z_m=X_0$. This verifies the Schur rigidity of $X_0$ in $X$.
\end{proof}

\begin{proposition}\label{product_case}
     Let $X_1,X_2\subset X=G/P_J$ be smooth Schubert varieties associated with the subdiagrams $\Gamma_{L_1}(I\cap L_1), \Gamma_{L_2}(I\cap L_2)$ of the marked Dynkin diagram $\Gamma_R(I)$ respectively, $L_1, L_2$ are disjoint such that $L_1\cup L_2$ is not connected, and $I\cap L_1,I\cap L_2$ consist of long roots. If the pairs $(G_{L_1}/(G_{L_1}\cap P_J), G/P_{J\cup (L_2 \cap I)}), (G_{L_2}/(G_{L_2}\cap P_{J}), G/P_{J\cup (L_1 \cap I)})$ are both Schur rigid, then $(G_{L_1}\times G_{L_2}/(G_{L_1}\times G_{L_2}) \cap P_J, G/P_J)$, which is associated with the subdiagram $\Gamma_{L_1\cup L_2}((L_1\cup L_2)\cap I)$, is also Schur rigid.
\end{proposition}

\begin{proof}
Let $Z$ be an irreducible reduced subvariety with $[Z]=r[G_{L_1}\times G_{L_2}/(G_{L_1}\times G_{L_2})\cap P_J]=r[X_1\times X_2]$. Consider the Mori contraction $\pi_{1}:  G/P_J \rightarrow G/P_{J\cup (L_2 \cap I)}$, we know that $\pi_{1}(Z)$ is irreducible reduced with $[\pi_{1}(Z)]\in \mathbb{R}^+[X_1]$. From the assumption, we know that there exists some $g_1 \in G$ such that \[\pi_{1}(Z)=g_1\cdot X_1=g_1\cdot (G_{L_1}/G_{L_1}\cap P_J).\] Similarly, for $\pi_{2}:  G/P_J \rightarrow G/P_{J\cup (L_1 \cap I)}$, we have
\[\pi_{2}(Z)=g_2\cdot X_2=g_2\cdot (G_{L_2}/G_{L_2} \cap P_J)\] for some $g_2\in G$. Since $L_1,L_2$ are disjoint, we have \[(J\cup(I\cap L_2))\cap (J\cup (I\cap L_1))\subset J,\]
and hence $(g^{-1}_1\pi_{1}, g_2^{-1}\pi_{2}): G/P_J \rightarrow G/P_{J\cup (L_2 \cap I)} \times G/P_{J\cup (L_1 \cap I)}$ is an embedding and the intersection of a $\pi_1$-fiber and a $\pi_2$-fiber is at most one point, we must have an isomorphism \[(g_1^{-1}\pi_{1}, g^{-1}_2\pi_{2})|_Z:Z \rightarrow G_{L_1}/(G_{L_1}\cap P_J) \times G_{L_2}/(G_{L_2}\cap P_J).\] This forces $r=1$. Then by Corollary \ref{subdiagram_rigid_marked_root_outside}, $Z$ is a $G$-translate of $G_{L_1} \times G_{L_2}/P_J$.
\end{proof}

\begin{proof}[Proof of Theorem \ref{Schur_subdiagram}]
First, we assume that $\Gamma_L$ is connected. When $|I\cap L|=1$, the conclusion follows from Lemma \ref{schur_pic_num_one}. When $|I\cap L|\geq 2$ and \eqref{eqn-intersection-Dynkin} holds, the conclusion follows from Proposition \ref{no_condition_E}. When $|I\cap L|\geq 2$ and \eqref{eqn-intersection-Dynkin} fails, we know that $G$ is of type $D$ or $E$ by Proposition \ref{prop-intersection-two-roots} and thus the conclusion follows easily from Proposition \ref{thm_with_condition_E}.

Finally, if we assume that $\Gamma_L$ is not connected, then the conclusion follows from Proposition \ref{product_case}, as if the product is not a fiber bundle over a projective space, then neither is each factor.
\end{proof}

\bigskip

Cong Ding(congding@szu.edu.cn)

\smallskip

School of Mathematical Sciences, Shenzhen University, Shenzhen 518060, China

\bigskip

Qifeng Li(qifengli@sdu.edu.cn)

\smallskip

School of Mathematics, Shandong University, Jinan 250100, China

\end{document}